\documentclass[12pt]{article}
\usepackage{amsmath}
\usepackage{amsthm}
\usepackage{amsfonts}
\usepackage{float}
\usepackage[dvips,final]{graphicx}
\usepackage{dcolumn}
\usepackage{comment} 
\usepackage{xspace}
\usepackage{url}
\usepackage{upgreek}
\usepackage{multicol}
\usepackage{pict2e,multirow}
\usepackage{verbatim}
\usepackage{wrapfig}
\usepackage{url}
\usepackage{tikz}
\usepackage[colorlinks=true,citecolor=blue,urlcolor=blue]{hyperref}
\usepackage{subfigure}
\usepackage[T1]{fontenc}
\usepackage[hmargin=.720in]{geometry}

\newtheorem{thm}{Theorem}

\newcommand{\be}{\begin{equation}}
\newcommand{\ee}{\end{equation}}

\newcommand{\bx}{\mathbf{x}}
\newcommand{\by}{\mathbf{y}}

\newcommand{\Dx}{\Delta x}

\newcommand{\bu}{\mathbf{u}}

\newcommand{\bw}{\mathbf{w}}

\newcommand{\Dt}{\Delta t}

\newcommand{\bbone}{\mathbf{1}}

\newcommand{\dt}{\Delta t}

\newcommand{\Complex}{\mathbb{C}}
\newcommand{\m}[1]{\mathbf{#1}}
\newcommand{\mA}{\m{A}}

\newcommand{\mAh}{\hat{\m{A}}}
\newcommand{\mD}{\m{D}}
\newcommand{\mS}{\m{S}}
\newcommand{\mT}{\m{T}}
\newcommand{\mR}{\m{R}}
\newcommand{\mP}{\m{P}}

\newcommand{\mI}{\m{I}}

\newcommand{\mzero}{\m{0}}

\renewcommand{\v}[1]{\boldsymbol{#1}}
\newcommand{\transpose}{^\mathrm{T}}

\newcommand{\vb}{\v{b}}

\newcommand{\ve}{\v{e}}

\newcommand{\vl}{\boldsymbol{l}}

\newcommand{\vd}{\v{d}}

\newcommand{\vf}{\v{f}}

\newcommand{\sspcoef}{\mathcal{C}}

\newcommand{\ceff}{\sspcoef_{\textup{eff}}}
\newcommand{\DtFE}{\Dt_{\textup{FE}}}

\newcommand{\lte}{\tau}
\newcommand{\ste}{\boldsymbol{\tau}}

\newcommand{\btheta}{\boldsymbol{\theta}}

\newcommand{\bff}{\mathbf{f}}

\newcommand{\bb}{\mathbf{b}}
\newcommand{\bbh}{\hat{\mathbf{b}}}
\newcommand{\bc}{\mathbf{c}}

\newcommand{\Rmin}{\overline{R}(\psi_1,\dots,\psi_k)}

\renewcommand{\v}[1]{\mathbf{#1}}

\title{Explicit Strong Stability Preserving Multistep Runge--Kutta Methods}

\author{%
Christopher Bresten\thanks{Mathematics Department, University of Massachusetts Dartmouth.
Supported by AFOSR grant number FA-9550-12-1-0224 and KAUST grant FIC/2010/05},
Sigal Gottlieb\footnotemark[1],
 Zachary Grant\footnotemark[1], \\
 Daniel Higgs\footnotemark[1], 
David I. Ketcheson\thanks{King Abdullah University of Science \& Technology (KAUST). },
and Adrian N\'emeth\thanks{Department of Mathematics and Computational Sciences, Sz\'echenyi Istv\'an University, Gy\H{o}r, Hungary}
}
\begin{document}
\maketitle


\bibliographystyle{siam}

\begin{abstract} 
High-order spatial discretizations with strong stability properties (such as monotonicity)
are desirable for the solution of hyperbolic PDEs. Methods may be compared
in terms of the strong stability preserving (SSP) time-step.
We prove an upper bound on the SSP coefficient of explicit multistep Runge--Kutta methods
of order two and above.  Order conditions and monotonicity conditions for such methods
are worked out in terms of the method coefficients.
Numerical optimization is used to find optimized explicit methods of up to five
steps, eight stages, and
tenth order. These methods are tested on
the advection and Buckley-Leverett equations, and the results
for the observed total variation 
diminishing and positivity preserving time-step are presented.
\end{abstract}

\section{Introduction\label{sec:intro}}

The numerical solution of hyperbolic conservation laws 
$ U_t +f(U)_x = 0$,
is complicated by the fact that the exact solutions may develop discontinuities. 
For this reason, significant effort has been expended on finding spatial discretizations
that can handle discontinuities \cite{SSPbook2011}.  
Once the spatial derivative is discretized, we obtain the  system of ODEs 
\begin{eqnarray}
\label{ode}
u_t = F(u),
\end{eqnarray}
where $u$ is a vector of approximations to $U$,  $u_j \approx U(x_j) $. This system of ODEs can then be 
evolved in time using standard methods.
The spatial discretizations used to approximate $f(U)_x$   are 
carefully designed so that when \eqref{ode}  is  evolved in time using the 
forward Euler  method the solution at time $u^n$ satisfies the strong stability property
\begin{align} \label{FEcond}
\| u^n + \dt F(u^{n}) & \| \leq \| u^n \| & \text{ under the step size restriction }
0 & \leq \dt \leq \DtFE.
\end{align}
Here and throughout, $\| \cdot \|$  represents a norm, semi-norm, or convex functional, 
determined by the design of the spatial discretization. For example, for total variation diminishing
methods the relevant strong stability property is in the total variation semi-norm,
while when using a positivity preserving limiter we are naturally interested in the
positivity of the solution.

The spatial discretizations satisfy the desired property when coupled with the forward Euler time discretization,
but in practice we want to use a higher-order time integrator rather than forward Euler, while still
ensuring that the  strong stability  property 
\begin{align}\label{monotonicity}
\| u^{n+1} \| \le \|u^n\|
\end{align}
is satisfied.

In \cite{shu1988b} it was observed that some Runge--Kutta methods can be decomposed into convex combinations
of forward Euler steps, and so any  convex functional  property satisfied by forward Euler will be {\em preserved}
by these higher-order time discretizations, generally under a different time-step restriction. This approach was used
to develop second and third order Runge--Kutta methods that preserve the strong stability properties of the spatial
discretizations developed in that work. In fact, this approach also guarantees that 
the intermediate stages in a Runge--Kutta method  satisfy the strong stability property as well.

For multistep methods, where  the solution value $u^{n+1}$ at time $t^{n+1}$  is computed from 
previous solution values $u^{n-k+1},\dots,u^n$,
we say that a  $k$-step numerical method is  {\em strong stability preserving} (SSP) if 
\be \label{kstepmonotonicity}
\|u^{n+1}\|\le\max\left\{\|u^n\|,\|u^{n-1}\|,\dots,\|u^{n-k+1}\|\right\}.
\ee
for any time-step \be \label{tstepcond} 0 \Dt\le \sspcoef \DtFE, \ee
(for some $\sspcoef>0$), assuming only that the spatial discretization satisfies \eqref{FEcond}.
An  explicit multistep method of the form
\begin{eqnarray} \label{lmmSO}
u^{n+1} & = & \sum_{i=1}^{k} \left( \alpha_i u^{n+1-i} +
\dt \beta_i F(u^{n+1-i}) \right)
\end{eqnarray}  
has $\sum_{i=1}^{k} \alpha_i =1$ for  consistency, so if all the 
coefficients are non-negative ($\alpha_i,\beta_i\ge0$) the method
can be written as  convex combinations of forward Euler steps:
\begin{eqnarray*}
u^{n+1} & = & \sum_{i=1}^{k} \alpha_i  \left( u^{n+1-i} +
\frac{\beta_i}{\alpha_i} \dt  F(u^{n+1-i}) \right).
\end{eqnarray*}  
Clearly,  if the forward Euler condition  \eqref{FEcond} holds
then the solution obtained by the multistep method (\ref{lmmSO})
is strong stability preserving  under the time-step restriction \eqref{tstepcond} with
$\sspcoef =  \min_{i} \frac{\alpha_i}{\beta_i} \DtFE$, 
(where if any $\beta_i$ is equal to zero, the corresponding  ratio is considered infinite)~\cite{shu1988b}.

The convex combination approach has also been applied to obtain sufficient
conditions for strong stability for 
{\em implicit} Runge--Kutta methods and {\em implicit} linear multistep methods.
Furthermore, it has be shown that these conditions are not only sufficient, but necessary as well
 \cite{ferracina2004, ferracina2005,higueras2004a, higueras2005a}.
Much research on SSP methods focuses on finding high-order time discretizations
with the largest allowable time-step  $\Dt \le \sspcoef \DtFE.$ Our aim is to maximize
the   {\em SSP coefficient} $\sspcoef$ of the method, relative to the number of function evaluations
at each time-step (typically the number of stages of a method).
For this purpose we define the {\em effective SSP coefficient} $\ceff = \frac{\sspcoef}{s}$
where $s$ is the number of stages. This value allows us to compare the efficiency of 
explicit methods of a given order.

Explicit  Runge--Kutta methods with positive SSP coefficients cannot be more than
fourth-order accurate \cite{kraaijevanger1991,ruuth2001}, while  explicit SSP linear
multistep methods of high-order accuracy must use very many steps, and therefore 
impose large storage requirements~\cite{SSPbook2011,lenferink1989}.  
These characteristics have led to the design of  explicit methods with multiple steps and multiple stages in  
the search for higher-order SSP methods with large effective SSP coefficients.
In \cite{gottlieb2001} Gottlieb et.~al.~considered a class of two-step, two-stage methods.
Huang \cite{huang2009} considered two-stage hybrid methods with many steps, and
found methods of up to seventh order (with seven steps) with reasonable SSP coefficients.
Constantinescu and Sandu \cite{constantinescu2009} found multistep Runge--Kutta with
up to four stages and four steps, with a focus on finding SSP methods with order up to four.
Multistep Runge--Kutta SSP methods with order as high as twelve have been developed
in \cite{nguyen2011} and numerous similar works by the same authors, using sufficient
conditions for monotonicity and focusing on a single set of parameters in each work.
Spijker \cite{spijker2007} developed a complete theory for strong stability preserving 
multi-step multi-stage methods
and found new second order and third order methods with optimal SSP coefficients.
In \cite{tsrk}, Spijker's theory (including necessary and sufficient conditions for 
monotonicity) is applied to two-step Runge--Kutta methods to develop two-step multi-stage 
explicit methods with optimized SSP coefficients. In the present work we 
present a general application of the same theory to multistep Runge--Kutta
methods with more steps.  We determine necessary and sufficient conditions for strong
stability preservation and prove sharp upper bounds on $\sspcoef$ for second order methods.
We also find and test optimized methods with up to five steps and up to tenth order.
The approach we employ ensures that the intermediate stages of each method also
satisfy a strong stability property.


In Section \ref{sec:sspglms} we extend the order conditions and SSP conditions from two step Runge--Kutta methods
 \cite{tsrk} to MSRK methods with  arbitrary numbers of steps and stages.
In Section \ref{sec:bounds} we recall an upper bound on $\sspcoef$ for general
linear methods of order one
and prove a new, sharp upper bound on $\sspcoef$ for general linear methods of order two. 
These bounds are important to our study because
the explicit MSRK methods we consider are a subset of the class of general linear methods.
In Section \ref{sec:optimal} we formulate and numerically solve the problem of determining
methods with the largest $\sspcoef$ for a given order and number of stages and steps.
We present the effective SSP coefficients of optimized methods of up to five steps and tenth order,
thus surpassing the order-eight barrier established in \cite{tsrk} for two-step methods.
Most of the methods we find have higher effective SSP coefficients than methods previously found,
though in some cases we had  trouble with the optimization subroutines for higher orders.
Finally, in Section \ref{sec:test} we explore how well these methods perform in practice, 
on a series of well-established test problems. We highlight the need for higher-order methods and the behavior 
of these methods in terms of strong stability and positivity preservation.

\section{SSP Multistep Runge--Kutta Methods\label{sec:sspglms}}
In this work we study methods in the class of multistep Runge-Kutta methods with optimal strong stability preservation properties.
These multistep Runge--Kutta methods are a simple generalization of Runge--Kutta methods to include the numerical 
solution at previous steps.  These methods are Runge--Kutta methods in the sense that they compute multiple stages based
on the initial input; however, they  use the previous $k$ solution values  $u^{n-k+1}, u^{n-k}, . . . , u^{n-1}, u^n$ to  compute the solution value $u^{n+1}$. 

A class of two-step Runge--Kutta methods was studied in \cite{tsrk}.
Here we study the generalization of that class to an arbitrary number of steps:
\begin{subequations} \label{eq:mrktypeII}
\begin{align}
y_1^n & = u^n \\
y_i^n & = \sum_{l=1}^{k} d_{il} u^{n-k+l} + \Dt\sum_{l=1}^{k-1} \hat{a}_{il} F(u^{n-k+l}) + \Dt\sum_{j=1}^{i-1} a_{ij} F(y_j^n) \; \; \; \;  2 \leq i \leq s \\
u^{n+1} & = \sum_{l=1}^{k} \theta_l u^{n-k+l} + \Dt\sum_{l=1}^{k-1} \hat{b}_{l} F(u^{n-k+l}) + \Dt\sum_{j=1}^s b_j F(y_j^n).
\end{align} \label{ksrk_b}
\end{subequations}
Here the values $u^{n-k+j}$  denote the previous steps and  $y^n_j$ are intermediate stages used to compute the next solution value $u^{n+1}$.  
The form \eqref{eq:mrktypeII} is convenient for identifying the computational cost of the method:  it is evident that $s$ new function evaluations
are needed to progress from $u^{n}$ to $u^{n+1}$. 

To study the strong stability preserving properties of method
\eqref{eq:mrktypeII}, we write it in the form \cite{spijker2007}
\begin{align}
\label{spijkerform-compact}
\bw & = \mS\bx + \Dt \mT \vf.
\end{align}
To accomplish this, we stack the last $k$ steps into a column vector:
\[\bx = \left[u^{n-k+1},u^{n-k+2};\dots,u^{n-1};u^n\right].\]
We define a column vector of length $k+s$ that contains these steps and the stages:
\[\bw = \left[u^{n-k+1};u^{n-k+2};\dots,u^{n-1};y_1=u^n;y_2;\dots;y_s;u^{n+1}\right] ,\]
and another column vector containing the derivative of each element of $\bw$:
\[\vf=\left[F\left(u^{n-k+1}\right);F\left(u^{n-k+2}\right);\dots;F\left(u^{n-1}\right),
F\left(y_1\right);\dots;F\left(y_s\right);F\left(u^{n+1}\right)\right]\transpose.\]
Here we have used the semi-colon to denote (as in MATLAB) vertical concatenation of vectors.
Thus, each of the above is a column vector.

Now the method (\ref{eq:mrktypeII}) can be written in the matrix-vector form \eqref{spijkerform-compact}
where the matrices $\mS$ and $\mT$ are
\begin{align} \label{msrk-spijker}
\mS & = \begin{pmatrix}\mI_{(k-1) \times (k-1)} \ \mzero_{1 \times (k-1)} \\ \mD \\\btheta\transpose \end{pmatrix} \ \ \ \
\mT   = \begin{pmatrix}\mzero & \mzero & 0 \\ \mAh & \mA & \mzero \\ \bbh\transpose & \bb\transpose & 0 \end{pmatrix}.
\end{align}
The matrices $\mD, \mA, \mAh$ and the vectors $\btheta,  \bbh, \bb$ 
contain the coefficients $d_{il}, \hat{a}_{il}, a_{ij}$ and $\theta_l, \hat{b}_l, b_j$
from (\ref{eq:mrktypeII}); note that 
the first row of $\mD$ is $(0,0,\dots,0,1)$ and the first row of
$\mA,\mAh$ is identically zero. Consistency requires that 
\begin{align*}
\sum_{l=1}^k \theta_l & = 1, \\
\sum_{l=1}^k d_{il} & = 1  & 1\le i \le s.
\end{align*}
We also assume that (see \cite[Section 2.1.1]{spijker2007})
\be \label{eq:Ssum}
\mS \ve = \ve,
\ee
where $\ve$ is a column vector with all entries equal to unity.
This condition is similar to the consistency conditions, and 
implies that every stage is consistent when viewed as a quadrature rule.

In the next two subsections we use representation \eqref{spijkerform-compact} to study monotonicity properties of the method \eqref{eq:mrktypeII}.
The results in these subsections are a straightforward generalization of the corresponding
results in \cite{tsrk}, and so
the discussion below is brief and the interested reader is referred to \cite{tsrk} for more detail. 

\subsection{A review of the SSP property for multistep Runge--Kutta methods}
To write \eqref{spijkerform-compact} as a linear combination of forward Euler steps, 
we add the term $r\mT \bw$ to both sides of \eqref{spijkerform-compact}, obtaining
\begin{align*}
    \left(\mI+r\mT\right) \bw & = \mS\bx 
                      + r\mT\left(\bw+\frac{\Dt}{r}\vf\right).
\end{align*}
We now left-multiply both sides by $\left(\mI+r\mT\right)^{-1} $ (assuming it exists) to obtain
\begin{align}  
    \bw & = (\mI+r\mT)^{-1} \mS\bx 
                      + r(\mI+r\mT)^{-1}\mT\left(\bw+\frac{\Dt}{r}\vf\right)  \nonumber \\ 
        & = \mR \bx + \mP\left(\bw+\frac{\Dt}{r}\vf\right), \label{canonical} 
\end{align}
where
\be   \label{canonical2}
\mP=r (\mI+r\mT)^{-1} \mT, \ \ \ \ \mR= (\mI+r\mT)^{-1} \mS= (\mI-\mP)\mS .
\ee

In consequence of the consistency condition \eqref{eq:Ssum},  
the row sums of $[\mR \ \mP]$ are each equal to one:
$$
\mR\ve+\mP\ve = (\mI-\mP)\mS\ve+\mP\ve = \ve - \mP\ve + \mP\ve = \ve.
$$
Thus if $\mR$ and $\mP$ have no negative entries, then each stage $w_i$ is
a convex combination of the inputs $x_j$ and the forward Euler quantities
$w_j+(\Dt/r) F(w_j)$. 
It is then simple to show (following \cite{spijker2007}) that
any strong stability property of the forward Euler method is 
preserved by the method \eqref{spijkerform-compact} under the time-step
restriction $\dt \leq \sspcoef(\mS,\mT)\DtFE$ where
$\sspcoef(\mS,\mT)$ is defined as
\begin{align*}
\sspcoef(\mS,\mT) & = \sup_{r}\left\{r : (\mI+r \mT)^{-1} \mbox{ exists and }  \mP \ge 0, \mR \ge 0 
\right\}.
\end{align*}
Hence the SSP coefficient of method
\eqref{canonical} is greater than or equal to $\sspcoef(\mS,\mT)$.
In fact, following \cite[Remark~3.2]{spijker2007}) we can conclude that
if the method is row-irreducible,  then the SSP coefficient 
is, in fact, exactly equal to $\sspcoef(\mS,\mT)$. (For the definition of row reducibility,
see \cite[Remark~3.2]{spijker2007}) or \cite{tsrk}).

\subsection{Order conditions}
In \cite{tsrk} we derived order conditions for methods of the form \eqref{eq:mrktypeII} with two steps. 
Those conditions extend in a simple way to method \eqref{eq:mrktypeII}
with any number of steps. For convenience, we rewrite \eqref{eq:mrktypeII} in the form
\begin{subequations} \label{eq:mrk}
\begin{align} 
\by^n & = \tilde{\mD} \bu^n + \Delta t  \tilde{\mA} \bff^n \\
u^{n+1} & = \btheta\transpose \bu^n + \Delta t \tilde{\bb}\transpose \bff^n
\end{align}
\end{subequations}
where 
\begin{align} 
\tilde{\mD} & = \begin{pmatrix}\mI_{(k-1) \times (k-1)} \ \mzero_{1 \times (k-1)} \\ \mD \\   \end{pmatrix} \ \ \ \
\tilde{\mA}   = \begin{pmatrix}\mzero & \mzero  \\ \mAh & \mA  \\  \end{pmatrix} \ \ \ \
\tilde{\bb} =  \begin{pmatrix} \bbh & \bb  \end{pmatrix}, 
\end{align}
and
$\by^n  = [u^{n-k+1}; u^{n-k+2}; \dots u^{n-1}; y_1^n;\dots;y_s^n], 
$ and $\bff^n  = F(\by^n)$  
are the vector of stage values and stage derivatives, respectively,
and  $\bu^n = [u^{n-k+1}, u^{n-k+2}, \dots, u^{n}]$ is the vector of previous step values.

The derivation of the order conditions closely follows Section 3 of \cite{tsrk}
with the following changes: (1) the vector $\vd$, is replaced by the matrix $\mD$;
(2) the scalar $\theta$ is replaced by the vector $\btheta$; and (3) the vector 
 $\vl = (k-1, k-2,\dots,1,0)^T$ appears in place of the number $1$ in the expression for the 
 stage residuals, which are thus:
\[ 
\ste_k  = \frac{1}{k!}\left(\bc^k - \tilde{\mD} (-\vl)^k\right) - \frac{1}{(k-1)!} \tilde{\mA} \bc^{k-1}, \ \ \ \ \ \ \ 
\lte_k  = \frac{1}{k!}\left(1 - {\btheta\transpose} (-\vl)^k\right) - \frac{1}{(k-1)!} \tilde{\vb}\transpose \bc^{k-1}, \]
where $\bc = \tilde{\mA}\ve- \tilde{\mD} \vl$  and exponents are to be interpreted element-wise.
The derivation of the order conditions is identical to that in \cite{tsrk} except for these changes.

A method is said to have stage order $q$ if $\ste_k$ and $\lte_k$ vanish for all $k\le q$.
The following result is a simple extension of Theorem 2 in \cite{tsrk}.

\begin{thm} \label{thm:stage-order}
Any irreducible MSRK method \eqref{ksrk_b} of order $p$ with positive SSP coefficient
has stage order at least $\lfloor \frac{p-1}{2} \rfloor$.
\end{thm}

Note that the approach used in \cite{tsrk}, which is based on the work of
Albrecht \cite{albrecht1996}, 
produces a set of order conditions that are {\em equivalent} to the set of conditions
derived using B-series. However, the two sets have different equations.
Albrecht's approach has two advantages over that based on B-series in the present context.  
First, it leads to algebraically simpler conditions that are almost identical in appearance to those
for one-step RK methods.
Second, it leads to conditions in which the residualts $\ste_k$ appear explicitly.  
As a result, very many of the order conditions are {\em a priori} satisfied by
methods with high stage order,
due to Theorem \ref{thm:stage-order}.  This simplifies the numerical optimization problem
that is formulated in Section \ref{sec:optimal}.

\section{Upper bounds on the SSP coefficient\label{sec:bounds}}
In this section we present upper bounds on the SSP coefficient
of general linear methods of first and second order.  These upper bounds
apply to all explicit multistep multistage methods, not just those of form \eqref{eq:mrktypeII}.
They are obtained by considering a relaxed optimization problem.
Specifically, we consider monotonicity and order conditions for methods
applied to linear problems only.


Given a function $\psi : \Complex \to \Complex$, let $R(\psi)$ denote
the {\em radius of absolute monotonicity}:
\begin{align}
R(\psi) & = \sup \{ r\ge0 \ \ | \ \ \psi^{(j)}(z)\ge 0 \mbox{ for all } z\in[-r,0]\}.
\end{align}
Here the $\psi^{(j)}(z)$ denotes the $j$th derivative of $\psi$ at $z$.
Any explicit general linear method applied to the linear, scalar ODE $u'(t)=\lambda u$
results in an iteration of the form
\begin{align}
u^{n+1} & = \psi_1(z)u_n + \psi_2(z) u_{n-1} + \cdots + \psi_k(z) u_{n-k+1},
\end{align}
where $z=\Dt\lambda$ and $\{\psi_1,\dots,\psi_k\}$ are polynomials of degree
at most $s$.  The method is strong stability preserving for linear problems
under the stepsize restriction $\Dt \le \Rmin \DtFE$ where
\begin{align}
\Rmin & = \min_i R(\psi_i).
\end{align}
The constant $\Rmin$ is commonly referred to as the {\em threshold factor} \cite{spijker1983}.  
We also refer to the {\em optimal threshold factor}
\begin{align}
R_{s,k,p} = \sup \left\{\Rmin \ \ | \ \ (\psi_1,\dots,\psi_k)\in\Pi_{s,k,p}\right\}
\end{align}
where $\Pi_{s,k,p}$ denotes the set of all stability functions of $k$-step,
$s$-stage methods satisfying the order conditions up to order $p$.
Clearly the SSP coefficient of any $s$-stage, $k$-step, order $p$ MSRK method
is no greater than the corresponding $R_{s,k,p}$.  Optimal values of $R_{s,k,p}$
are given in \cite{ketcheson2008}.

The following result is proved in Section 2.3 of \cite{gottlieb2009}.
\begin{thm}
The threshold factor of a first-order accurate explicit $s$-stage general
linear method is at most $s$.
\end{thm}
Methods consisting of $s$ iterated forward Euler steps achieve this bound
(with both threshold factor and SSP coefficient equal to $s$).
Clearly it provides an upper bound on the threshold factor and SSP coefficient
also for methods of higher order.  For second
order methods, a tighter bound is given in the next theorem.  We will see
in Section \ref{sec:optimal} that it is sharp, even over the smaller
class of MSRK methods.

\begin{thm} \label{thm:2nd}
For any $s \geq 0, k > 1$ the optimal threshold factor for explicit $s$-stage,
$k$-step, second order general linear methods is
\begin{equation} \label{rsk2}
R_{s,k,2}:=\frac{(k-2)s+\sqrt{(k-2)^2s^2+4s(s-1)(k-1)}}{2(k-1)}.
\end{equation}
\end{thm}
\begin{proof}
It is convenient to write the stability polynomials in the form
\begin{align}
    \psi_i = \sum_j \gamma_{ij} \left(1+\frac{z}{r}\right)^j
\end{align}
where we assume $r\in[0,\Rmin]$, which implies 
\begin{align} \label{gammapos}
\gamma_{ij}\ge0.
\end{align}
The conditions for second order accuracy are:
\begin{subequations} \label{2oc}
\begin{align}
& \sum\limits_{i=1}^{k}\sum\limits_{j=0}^s\gamma_{ij} = 1, \label{0thorder} \\
& \sum\limits_{i=1}^{k}\sum\limits_{j=0}^s\gamma_{ij}(j+(k-i)r) = kr, \label{1storder} \\
& \sum\limits_{i=1}^{k}\sum\limits_{j=0}^s\gamma_{ij}((k-i)^2r^2+2(k-i)jr+j(j-1)) = k^2r^2.  \label{2ndorder}
\end{align}
\end{subequations}
We will show that conditions \eqref{gammapos} and \eqref{2oc} cannot be
satisfied for $r$ greater than the claimed value \eqref{rsk2}, which we denoted 
in the rest of the proof simply by $R_2$.

By way of contradiction, suppose $r>R_2$.
Multiply \eqref{1storder} by $kr$ and subtract \eqref{2ndorder} from the result
to obtain
\begin{equation}
\sum\limits_{i=1}^{k}\sum\limits_{j=0}^s\gamma_{ij}(i(k-i)r^2-(k-2i)jr-j(j-1))=0.
\label{2ndorderb}
\end{equation}
Let us find the maximal root of this equation, which is an upper bound on $r$.
We introduce the following notation:
\begin{subequations}
\begin{align}
a(\gamma)&=+\sum\limits_{i=1}^{k}\sum\limits_{j=0}^s\gamma_{ij}i(k-i),\\
b(\gamma)&=-\sum\limits_{i=1}^{k}\sum\limits_{j=0}^s\gamma_{ij}(k-2i)j,\\
c(\gamma)&=-\sum\limits_{i=1}^{k}\sum\limits_{j=0}^s\gamma_{ij}j(j-1).
\end{align}
\end{subequations}

{\bf Case 1: $a(\gamma)=0$}.
In this case we have $\gamma_{ij}=0$ for all $i\neq k$, so \eqref{2ndorderb}
simplifies to
\begin{equation}
\sum\limits_{j=0}^s\gamma_{kj}j(kr-(j-1))=0.
\end{equation}
This implies that either $\gamma_{kj}=0$ for $j\ne0$ or that $r\le (s-1)/k$.
The first option fails to satisfy \eqref{1storder}, while the second
contradicts our assumption $r>R_2$.

{\bf Case 2: $a(\gamma)\ne0$}.
The largest root always exists due to the positivity of $a(\gamma)$ and the
nonpositivity of $c(\gamma)$, and it can be expressed as
\begin{equation}
r(\gamma)=\frac{-b(\gamma)}{2a(\gamma)}+\sqrt{\left(\frac{-b(\gamma)}{2a(\gamma)}\right)^2+\frac{-c(\gamma)}{a(\gamma)}},
\label{rmax}
\end{equation}
which simplifies to the desired $r=R_2$ in case 
\begin{align} \label{delta}
\gamma_{1s}=1,\quad\gamma_{ij}=0 \textrm{ for all } (i,j) \neq (1,s).
\end{align}
We now show that any positive coefficients $\gamma_{ij}$ can be transformed into
the choice \eqref{delta} without decreasing the largest root of \eqref{2ndorderb}.

Differentiating $r(\gamma)$ with respect to $\gamma_{kj}$ yields
\begin{eqnarray}
\frac{\partial}{\partial{\gamma_{kj}}} r(\gamma)&=&
\frac{-kj}{2a(\gamma)}+\frac{2b(\gamma)kj + 4a(\gamma)j(j-1)}{4a(\gamma)\sqrt{b(\gamma)^2-4a(\gamma)c(\gamma)}} \notag\\
&=& \frac{-r(\gamma)kj+j(j-1)}{\sqrt{b(\gamma)^2-4a(\gamma)c(\gamma)}},
\end{eqnarray}
which is non-positive by our assumption $r>R_2$.
Thus the largest root of \eqref{2ndorderb} will not decrease if we set
\begin{equation}
\gamma_{kj}:=0
\end{equation}
and then renormalize all the remaining $\gamma_{ij}$ so that \eqref{0thorder} holds.
Next we apply the transformation
\begin{subequations}
\begin{align}
\gamma_{ij} & :=\gamma_{ij}+\gamma_{k-i,j} &\textrm{ for all } 1 \leq i < \frac{k}{2},\\
\gamma_{ij} & :=0 &\textrm{ for all } \frac{k}{2} < i < k.
\end{align}
\end{subequations}
which leaves $a(\gamma)$ and $c(\gamma)$ invariant, ensures $b(\gamma)$ is nonpositive and
increases its absolute value, thus increases the largest root.
Now only negative terms contribute to $b(\gamma),c(\gamma)$ and only
positive terms contribute to $a(\gamma)$.  It follows that for fixed $(i,j) \neq (1,s)$ the transformation
\begin{subequations}
\begin{align}
\gamma_{1s} &:=\gamma_{1s} + \gamma_{ij},\\
\gamma_{ij} &:= 0
\end{align}
\end{subequations}
increases the largest root as it decreases the positive
$a(\gamma)$ and increases the absolute value of the nonpositive $b(\gamma),c(\gamma)$.
Applying the transformation for all $i,j$ we obtain \eqref{delta}.

We have shown that the claimed value is an upper bound on $R_2$.
This upper bound is achieved by taking
\begin{equation}
\gamma_{1s}=\frac{kR_2}{s-R_2+kR_2},\quad \gamma_{k0}=\frac{s-R_2}{s-R_2+kR_2}, 
\quad \gamma_{ij}=0 \textrm{ for all } (i,j) \notin \left\{(1,s),(k,0)\right\},
\end{equation}
which not only satisfy condition \eqref{2ndorderb} but also \eqref{gammapos} since $R_2 < s$.
\end{proof}


\section{Optimized explicit MSRK methods\label{sec:optimal}}
In this section we present an optimization problem for finding MSRK methods 
with the largest possible SSP coefficient. This optimization problem is
implemented in a MATLAB code and solved using the {\tt fmincon} function for optimization 
(code is available at our website \cite{sspsite}).
This implementation recovers the known optimal methods of first and second order
mentioned above.
For high order methods with large numbers of stages and steps, numerical
solution of the optimization
problem is difficult due to the number of coefficients and constraints. 
Despite the extensive numerical optimization searches,
we do not claim that all of the methods found are truly optimal; 
we refer to them only as {\em optimized}. 
Some of the higher-order methods {\em are} known to be optimal because
they achieve known upper bounds based on a relaxation of the optimization 
problem (presented in Section \ref{sec:bounds}) or on certified computations in
earlier work \cite{constantinescu2009}. 

In Section 4.2  we present the effective SSP coefficients of the optimized methods.
The coefficients $d_{il} ,  \hat{a}_{il} , a_{ij} ,  \theta_l , \hat{b}_{l} $ and  $b_j$ can be downloaded 
(as MATLAB files) from \cite{sspsite}.
The SSP coefficients of methods known to be optimal are printed in boldface in the corresponding
tables. The coefficients of methods that are known not to be optimal (e.g. when better methods
have been found in the literature) are printed in the table in a light grey. We chose to include these
to show the issues with the performance of the optimizer. We discuss these issues in
the relevant sections below.

A major issue in the implementation and the performance of the optimized time integrators is the choice
of starting methods to obtain the initial $k$ step values. Typically exact 
values are not available, and we recommend the use of 
many small steps of a lower order SSP method to generate the starting values. 
A discussion of starting procedures appears in \cite{tsrk}.

\subsection{The optimization problem}
Based on the results above, the problem of finding optimal SSP multistep Runge--Kutta
methods can be formulated algebraically.
We wish to find coefficients $\mS$ and $\mT$ (corresponding to
\eqref{msrk-spijker}) that maximize the value of $r$ subject to the following conditions:
\begin{enumerate}
\item $ (\mI+r \mT)^{-1} $ exists
\item $r (\mI+r\mT)^{-1} \mT \ge 0 $ and $ (\mI+r\mT)^{-1} \mS \ge 0$, where the inequalities are
 understood component-wise.
\item $\mS$ and $\mT$ satisfy the relevant order conditions.
\end{enumerate}  
This is a non-convex, nonlinear constrained optimization problem in many variables.
The second constraint above implies some useful bounds on the coefficients.
Extending Theorem 3 of \cite{tsrk}, one finds that if
method \eqref{eq:mrktypeII} has positive SSP coefficient then
\begin{subequations}
\begin{align}
\label{condition_i} \mzero & \le{\mD}\le \bbone, & 0 & \le\theta\le1, \\
\label{condition_iA} {\mA} & \ge\mzero, & {\mAh} & \ge\mzero, \\
\label{condition_ib} {\vb} & \ge\mzero  & {\bbh} & \ge\mzero .
\end{align}
\end{subequations}

This problem was used to formulate a MATLAB optimization code that uses
{\tt fmincon}. We ran this extensively, and when needed used methods with lower number of steps 
as starting values.  We note that for a large number of coefficient and constraints, this optimization
process was slow and seemed to get stuck in local minima.  

\vspace{-.2in}
\subsection{Effective SSP coefficients of the optimized methods}
We now discuss the optimized SSP coefficients among methods with prescribed
order, number of stages, and number of steps.  For a given order, the SSP coefficient is larger for methods with
more stages, and usually the effective SSP coefficient is also larger.
Comparing optimized SSP coefficients among classes of methods with the same
number of stages and order, but different number of steps, we see the following
behavior:
\begin{itemize}
    \item For methods of even order, the SSP coefficient increases monotonically with 
            $k$, and the marginal increase from $k$ to $k+1$ is smaller for larger $k$.
    \item For  methods of odd order up to five, for a large enough number of  stages
     there exists $k_0$ such that optimized methods never
            use more than $k_0$ steps (hence the optimized SSP coefficient remains the 
            same as the allowed number of steps is increased beyond $k_0$).  The value
            of $k_0$ depends on the order and number of stages. 
\end{itemize}
This behavior seems to generalize that seen for multistep methods \cite{lenferink1989}.
The behavior described for odd orders is observed here up to order five. 
Since the value of $k_0$ increases with $p$, we expect that
a study including larger $k$ values would show the same behavior for optimized methods of higher
(odd) order as well.
Overall, the effective SSP coefficient tends to increase more quickly with
the number of stages than with the number of steps.

Where relevant, we compare the methods we found to those of 
Constantinescu \cite{constantinescu2009}, Huang \cite{huang2009}, and Vaillancourt \cite{VaillancourtAPNUM2011, VaillancourtJSC2012, VaillancourtJCAM2014}.

\vspace{-.15in}
\subsubsection{Second-order methods} 
The second-order methods were first found by the numerical optimization procedure above.
We observed that the coefficients of the optimal second-order methods have a clear structure,
which we were 
then able to generalize and prove optimal in Theorem \ref{thm:2nd} above.

\begin{wraptable}[5]{r}{3.5in} \vspace{-.27in}
\caption{$\ceff$ for second-order methods}  \label{tab:2ord}
\begin{tabular}{|c|c|c|c|c|} \hline 
$s \backslash k$ &2&3&4&5\\ \hline
2 & {\bf 0.70711} & {\bf 0.80902} & {\bf 0.86038} & {\bf 0.89039}\\ \hline
3 & {\bf 0.81650} & {\bf 0.87915} & {\bf 0.91068} & {\bf 0.92934}\\ \hline
4 & {\bf 0.86603} & {\bf 0.91144} & {\bf 0.93426} & {\bf 0.94782}\\ \hline
5 & {\bf 0.89443} & {\bf 0.93007} & {\bf 0.94797} & {\bf 0.95863}\\ \hline
6 & {\bf 0.91287} & {\bf 0.94222} & {\bf 0.95694} & {\bf 0.96573}\\ \hline
7 & {\bf 0.92582} & {\bf 0.95076} & {\bf 0.96327} & {\bf 0.97074}\\ \hline
8 & {\bf 0.93541} & {\bf 0.95711} & {\bf 0.96798} & {\bf 0.97448}\\ \hline
\end{tabular} 
\end{wraptable}
Let $Q = 2(k-1)R_{s,k,2}$.
The non-zero coefficients of these methods are:
\begin{align*}
d_{ik} & = 1 & 1\le i \le s, \\
    b_j = \beta & := \frac{kQ}{s(k-1)\left(2(s-1)+Q \right)} & 1\le j \le s, \\
    a_{ij} = & = \frac{1}{R_{s,k,2}} & 1 \le j < i \le s \\
    \theta_k & = \frac{k - \beta s}{k-1} & \theta_1 = 1-\theta_k.
\end{align*}

These methods have  $\sspcoef = R_{s,k,2}$, which is proven optimal in Theorem \ref{thm:2nd} above.
In Table \ref{tab:2ord} these values appear for $s=2, . . ., 8$ and $k=2, . . .,5$.
While the second-order methods are not so useful from a practical point of view, as many good low-order SSP methods
are known, they are of great interest because the
the optimal SSP coefficient among 2nd-order methods with $k$ steps and $s$ stages is
an upper bound on the SSP coefficient for higher-order methods with the same values of $k$ and $s$.


\begin{wraptable}[10]{r}{3.5in} 
\vspace{-.65in}
\caption{$\ceff$ for third-order methods} 
 \label{tab:3ord}
\begin{tabular}{|l|l|l|l|l|} \hline
$s\backslash k$ &2&3&4&5\\\hline
2 & {\bf 0.36603} & {\bf 0.55643} & {\bf 0.57475} & 0.57475\\ \hline
3 & {\bf 0.55019} & {\bf 0.57834} & {\bf 0.57834} & 0.57834\\ \hline
4 & {\bf 0.57567} & {\bf 0.57567} & {\bf 0.57567} & 0.57567\\ \hline
5&0.59758&0.59758&0.59758&0.59758\\ \hline
6&0.62946&0.62946&0.62946&0.62946\\ \hline
7&0.64051&0.64051&0.64051&0.64051\\ \hline
8&0.65284&0.65284&0.65284&0.65284\\ \hline
9&0.67220&0.67220&0.67220&0.67220\\ \hline
10&0.68274&0.68274&0.68274&0.68274\\ \hline
\end{tabular}
\end{wraptable}
\subsubsection{Third-order methods}
The effective SSP coefficients of optimized third-order methods are shown in Table \ref{tab:3ord}
and plotted in Figure \ref{fig:3rdO}.
All methods with four or more stages turn out to be two-step methods
(i.e., $k_0=2$ for this case).
For $s=3$, there is no advantage to increasing the number of steps beyond $k_0=3$,
and for $s=2$, $k_0=4$. Note that although we report only values up
to five steps, this pattern was verified up to $k=8$. 
All methods up to $k=4, s=4$ are optimal (to two decimal places)
according to results of \cite{constantinescu2009}, and the $\ceff$ values for 
$(s,k)=(2,2),(3,2),(2,3)$ are provably optimal because they
achieve the optimal value $R_{s,k,3}$, as described above.

\vspace{-.15in}
\subsubsection{Fourth-order methods} \vspace{-.1in}
Effective coefficients are given in Figure \ref{fig:4thO} and 
Table \ref{tab:4ord}.
All methods up to $k=4, s=4$ are optimal (to two decimal places)
according to the certified optimization performed in \cite{constantinescu2009}. 
The $(2,5,4)$ method we found has an SSP coefficient that matches that of 
\cite{huang2009}.

\vspace{-.12in}
\begin{table}[hb]
\begin{minipage}{3.5in}
\centering
\caption{$\ceff$ for fourth-order methods}  \label{tab:4ord}
\begin{tabular}{|l|l|l|l|l|} \hline
$s\backslash k$ &2&3&4&5\\\hline
2 & {\bf --      } & {\bf 0.24767} & {\bf 0.34085} & 0.39640\\\hline
3 & {\bf 0.28628} & {\bf 0.38794} & {\bf 0.45515} & 0.48741\\ \hline
4 & {\bf 0.39816} & {\bf 0.46087} & {\bf 0.48318} & 0.49478\\ \hline
5&0.47209&0.50419&0.50905&0.51221\\ \hline
6&0.50932&0.51214&0.51425&0.51550\\ \hline
7&0.53436&0.53552&0.53610&0.53646\\ \hline
8&0.56151&0.56250&0.56317&0.56362\\ \hline
9&0.58561&0.58690&0.58871&0.58927\\ \hline
10&0.61039&0.61415&0.61486&0.61532\\ \hline
\end{tabular}
\end{minipage}
\begin{minipage}{3.5in}
\centering
\caption{$\ceff$ for fifth-order methods} \label{tab:5ord}
\begin{tabular}{|l|l|l|l|l|} \hline
$s\backslash k$ &2&3&4&5\\\hline
2&--&--&0.18556&0.26143\\\hline
3&--&0.21267&0.33364&0.38735\\\hline
4&0.21354&0.34158&0.38436&0.39067\\\hline
5&0.32962&0.38524&0.40054&0.40461\\\hline
6&0.38489&0.40386&0.40456&0.40456\\\hline
7&0.41826&0.42619&0.42619&0.42619\\\hline
8&0.44743&0.44743&0.44743&0.44743\\\hline
9&0.43794&0.43806&0.43806&0.43806\\\hline
10&0.42544&0.43056&0.43098&0.43098\\\hline
\end{tabular}
\end{minipage}
\end{table}

\vspace*{-.1in}

\begin{figure}[p] \label{fig:OptimalCeff}
\subfigure[Third-order methods\label{fig:3rdO}]{
\includegraphics[width=0.4\textwidth]{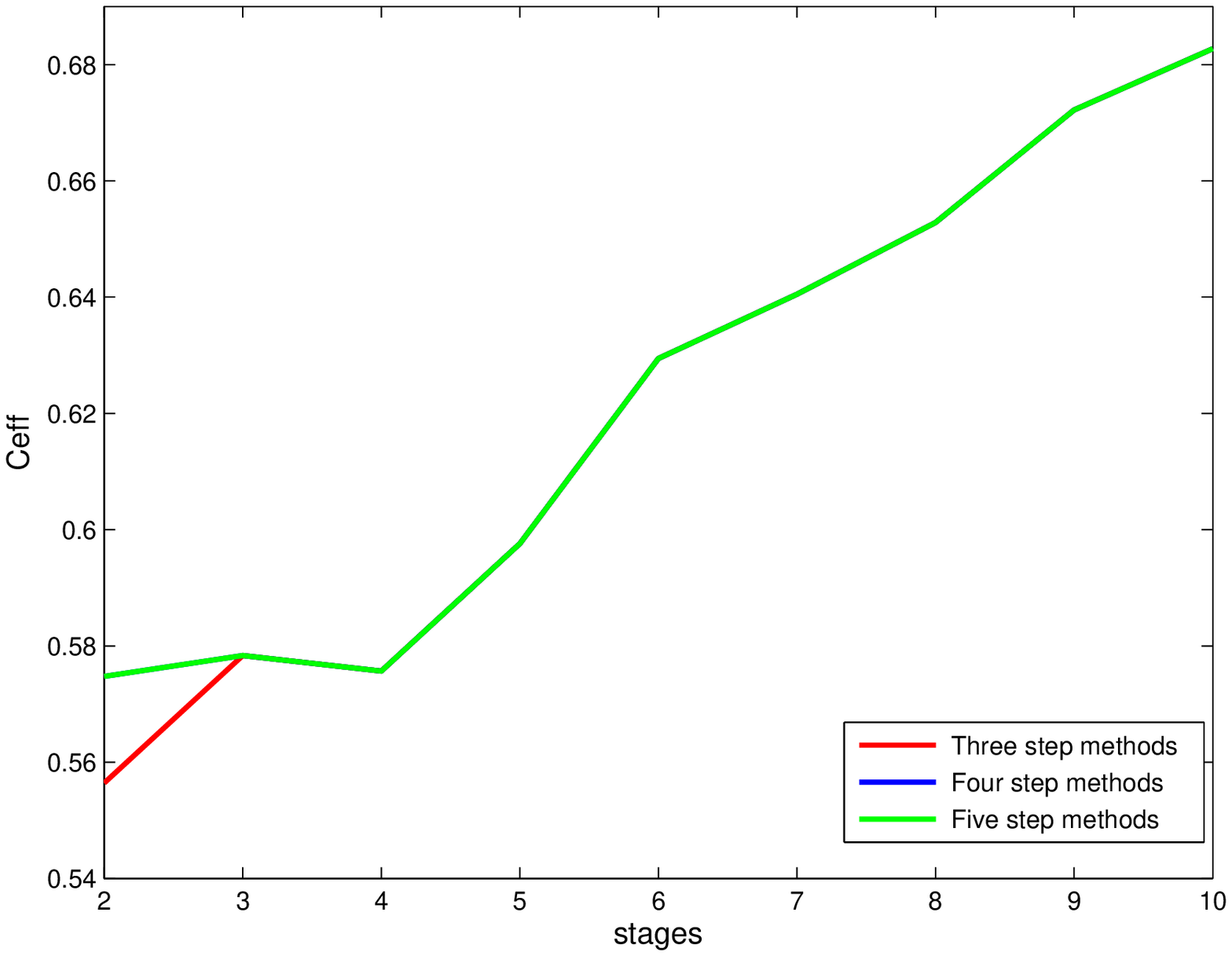}}
\subfigure[Fourth-order methods\label{fig:4thO}]{
\includegraphics[width=0.4\textwidth]{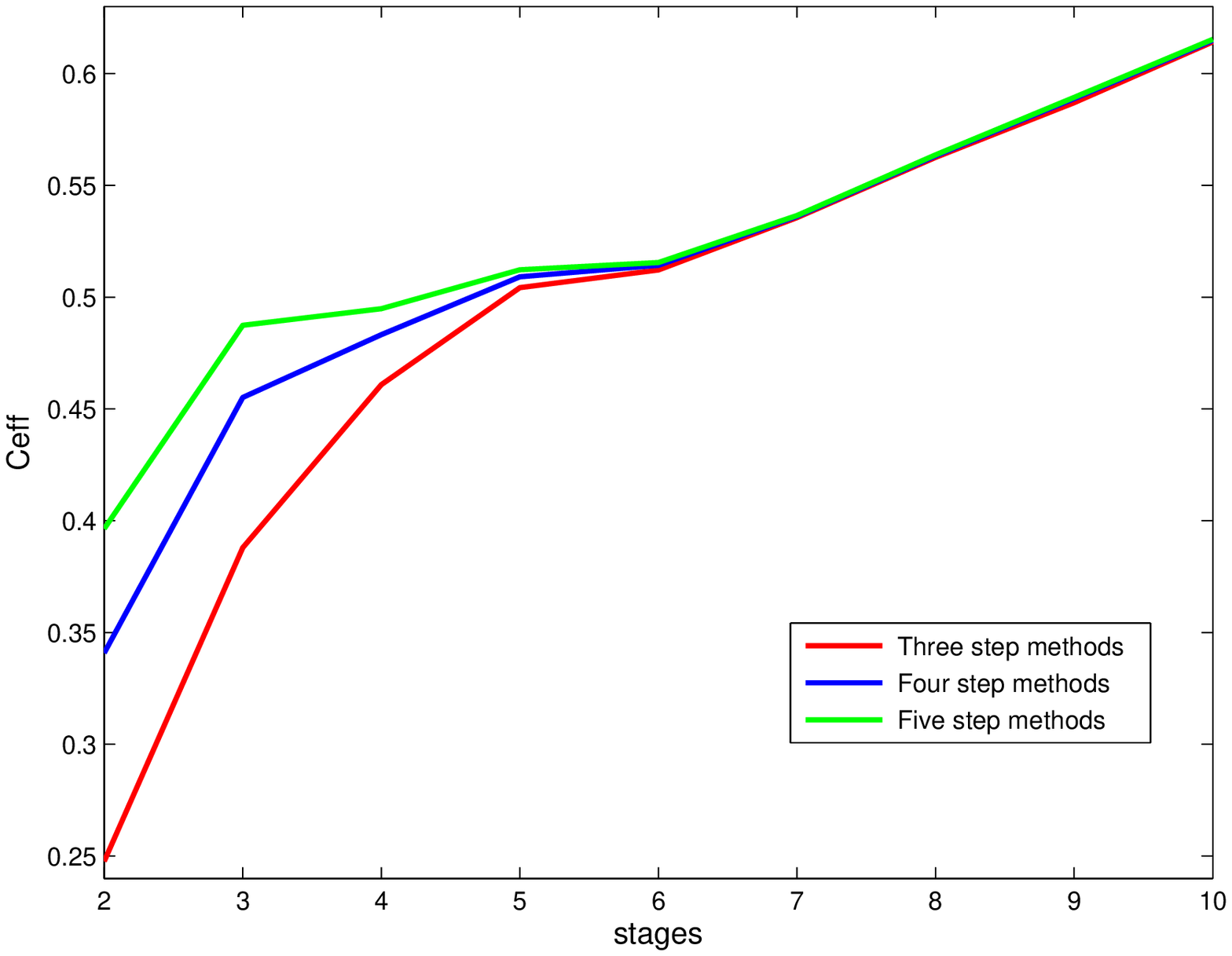}}
\subfigure[Fifth-order methods\label{fig:5thO}]{
\includegraphics[width=0.4\textwidth]{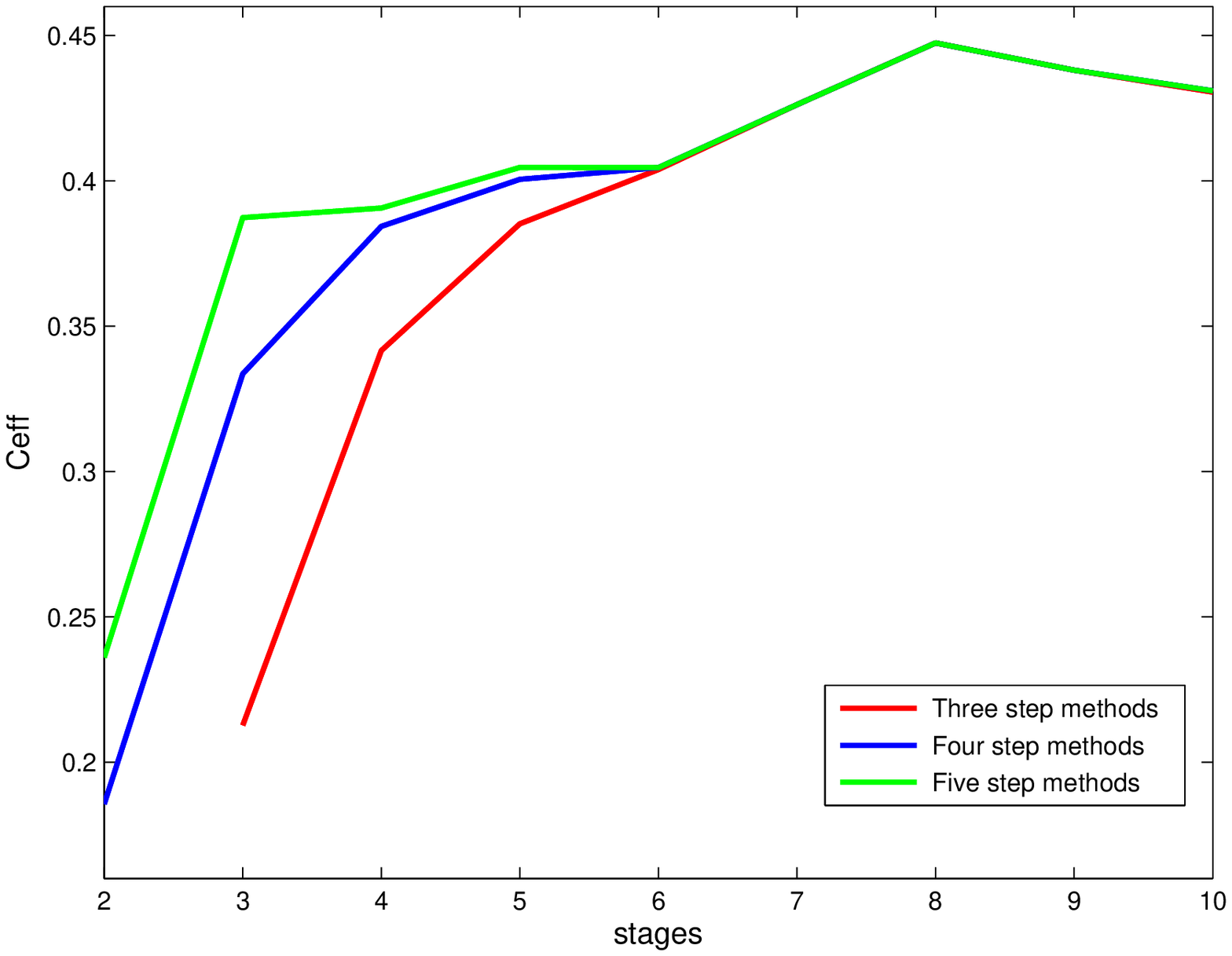}  }
\subfigure[Sixth-order methods\label{fig:6thO}]{
\includegraphics[width=0.4\textwidth]{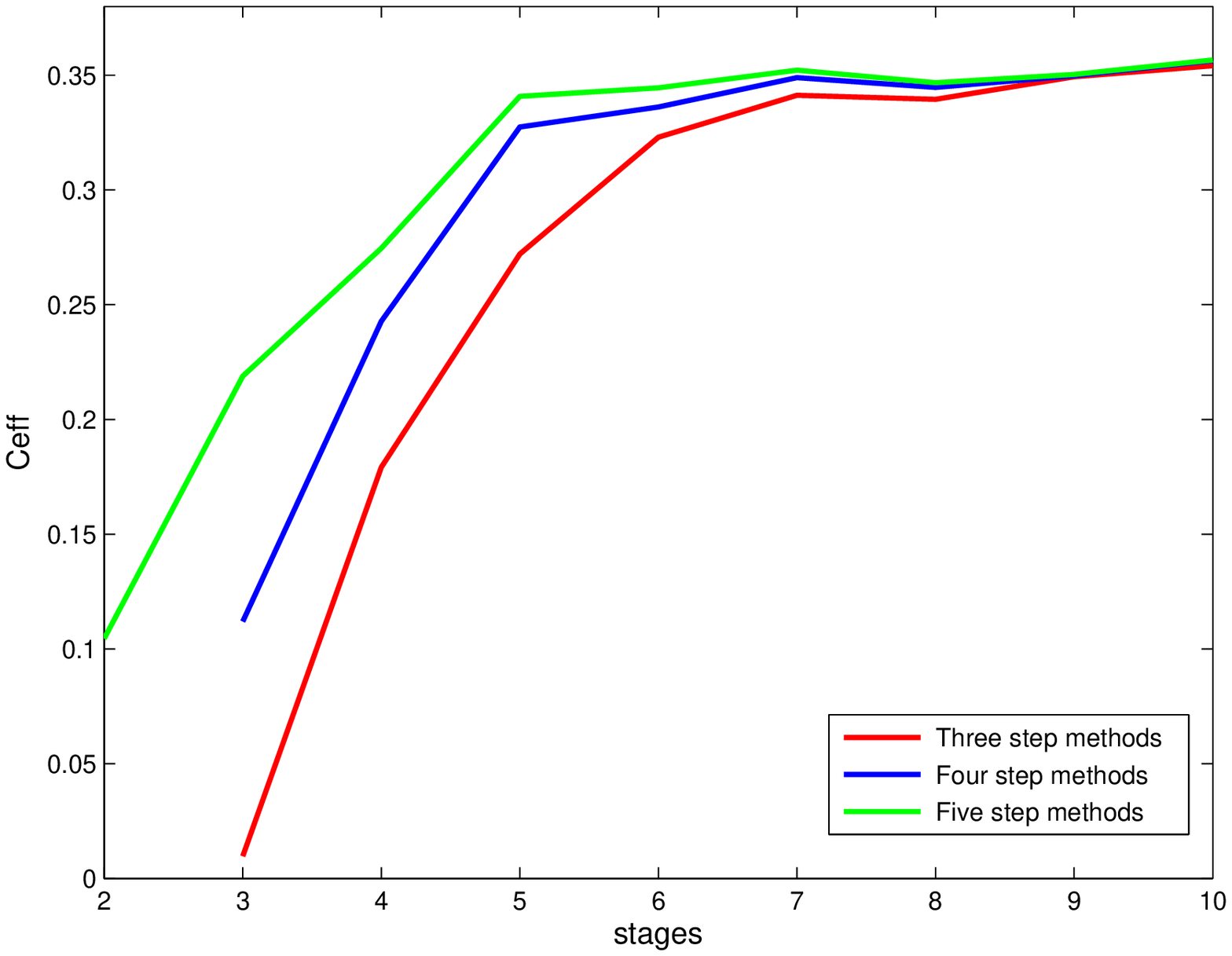}  }
\subfigure[Seventh-order methods\label{fig:7thO}]{
\includegraphics[width=0.4\textwidth]{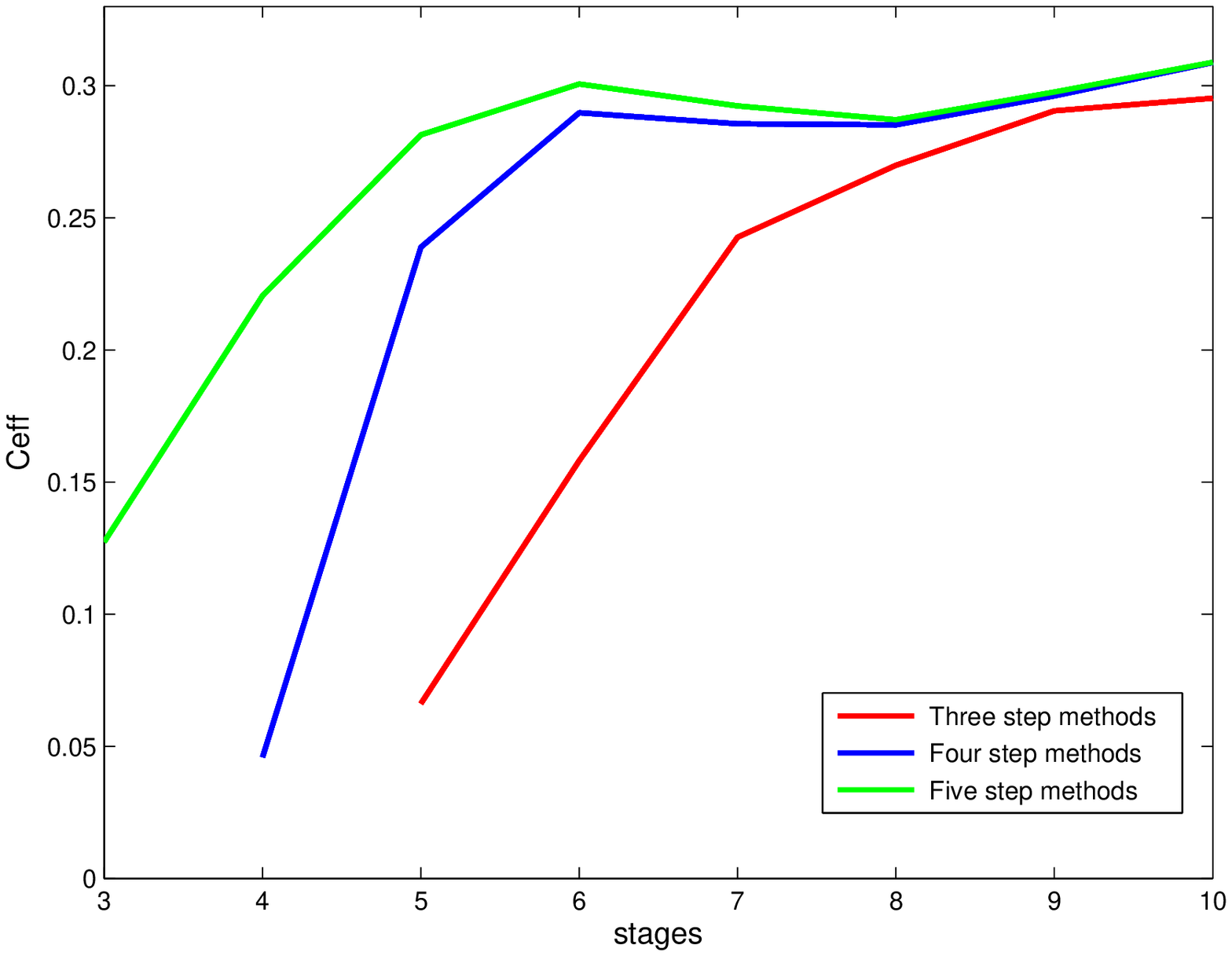}  }
\subfigure[Eighth-order methods\label{fig:8thO}]{\hspace{1.2in}
\includegraphics[width=0.4\textwidth]{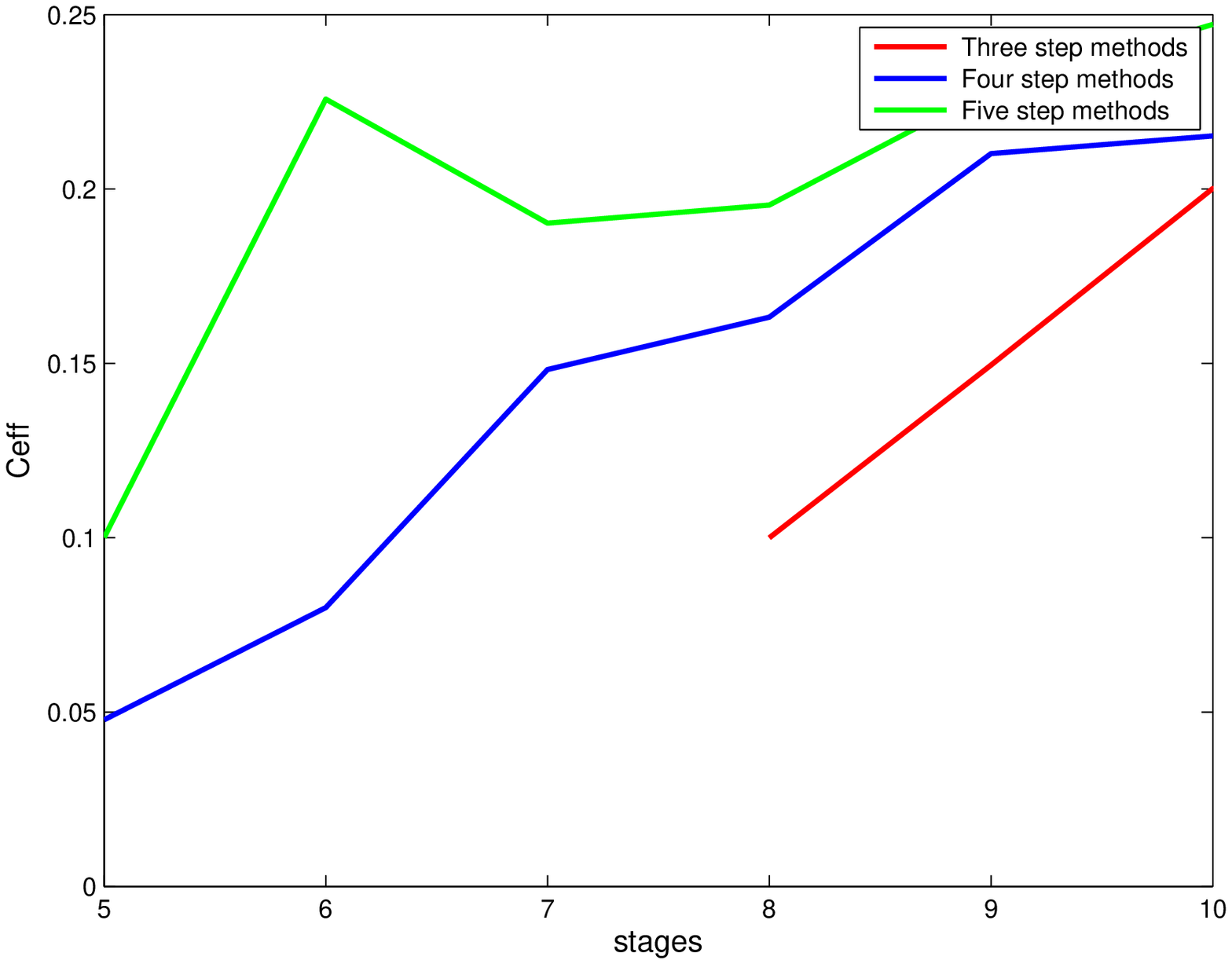}  }
\caption{Effective SSP coefficients of optimized methods.}
\end{figure}

\vspace*{-.2in}

\subsubsection{Fifth-order methods} The effective SSP coefficients
of the fifth-order methods are displayed in Figure \ref{fig:5thO} and 
Table \ref{tab:5ord}.
Although the optimized SSP coefficient is a strictly increasing function of the
number of stages, in some cases the effective SSP coefficient decreases. 
Our
 $(s,k)=(2,4)$ and $(s,k)=(2,5)$ methods have
effective SSP coefficients that match the ones in  \cite{huang2009}.
Our $(s,k)=(8,2), (3,3), (3,4),(3,5),$ and $(7,3)$ methods have effective SSP coefficient that match
those in  \cite{VaillancourtJCAM2014,VaillancourtAPNUM2011,VaillancourtJSC2012}.

\begin{wraptable}[9]{r}{3.5in} \vspace{-.8in}
\caption{$\ceff$ for sixth-order methods}
 \label{tab:6ord}
\begin{tabular}{|l|l|l|l|l|}
\hline
$s\backslash k$ &2&3&4&5\\\hline
2&--&--&--&  0.10451\\ \hline
3&--&0.00971& {\color{gray} 0.11192} &{\color{gray} 0.21889}\\\hline
4&--&0.17924&0.27118&0.31639\\\hline
5&--&0.27216&0.32746&0.34142\\\hline
6&0.09928&0.32302&0.33623&0.34453\\\hline
7&0.18171&0.34129&0.34899&0.35226\\\hline
8&0.24230&0.33951&0.34470&0.34680\\\hline
9&0.28696&0.34937&0.34977&0.35033\\\hline
10&0.31992&0.35422&0.35643&0.35665\\\hline
\end{tabular}
\end{wraptable}

\vspace*{-.1in}
\subsubsection{Sixth-order methods} 
Effective SSP coefficients of optimized sixth-order methods are given in Figure
\ref{fig:6thO} and Table \ref{tab:6ord}.
Once again, the effective SSP coefficient occasionally decreases with increasing stage number.
Our $(s,k)=(2,5)$ method  has an effective SSP coefficient that matches the one in 
\cite{huang2009}, and our values for $(s,k)= (8,3), (8,4),$ and $(8,5)$ improve upon the values obtained 
in \cite{VaillancourtJCAM2014}. Our values for $(s,k)=(7,3), (7,4)$ match those of 
\cite{VaillancourtJSC2012} and our $(s,k)=(7,5)$ value
improves on that in \cite{VaillancourtJSC2012}. The $(s,k)=(3,4), (3,5)$ values illustrate the challenges in
using our general numerical optimization formulation for this problem: we were not able to match the methods in \cite{VaillancourtAPNUM2011}
from a "cold start" (with random initial guesses). However, converting their methods to our form we were able to replicate their results while tightening the optimizer parameters {\tt TolCon,
TolFun} and {\tt TolX} from $10^{-12}$ in their work to $10^{-14}$.
This suggests that  the approach used in \cite{VaillancourtAPNUM2011} which focuses on one set of parameters
at a time may make the optimization problem more manageable. However, this same approach was used in
 \cite{VaillancourtJSC2012} and led to a  $(s,k)=(7,5)$ method that had a smaller SSP coefficient
than that found with our approach. 


\subsubsection{Seventh order methods} The effective SSP coefficients
for the seventh order case show consistent increase as both the 
steps and stages increase. There is more benefit to increasing 
stages rather than steps once the number of steps is large enough,
though for small $k$ relative to $s$ an an increase in steps is preferable.
The behavior of the effective SSP coefficient is also summarized in
Figure \ref{fig:7thO} and Table \ref{tab:7ord}. Compared to the seven-step two-stage method in
\cite{huang2009}, which has $\sspcoef=0.234$ and $\ceff = 0.117$, our five step methods
with $s\geq 3$, four step with $k\geq 5$,  three step with $k \geq 6$ and two step 
with $k \geq 9$ all have larger effective SSP coefficient. Our $(7,4), (7,5), (3,5)$ methods
have SSP coefficients that match those in \cite{VaillancourtJSC2012} and \cite{VaillancourtAPNUM2011},
while our $(7,3)$ and $(8,3), (8,4),(8,5)$ have larger SSP coefficients that those in  \cite{VaillancourtJSC2012}
and \cite{VaillancourtJCAM2014}.

\vspace{-.25in}
\begin{table}[ht]
\begin{minipage}{3.5in}
\centering
\caption{$\ceff$ for seventh order methods}  \label{tab:7ord}
\begin{tabular}{|l|l|l|l|l|}
\hline
\multicolumn{5}{|c|}{$\ceff$ for seventh order methods} \\ \hline
$s\backslash k$ &2&3&4&5\\\hline
2&--&--&--&--\\\hline
3&--&--&--&0.12735\\\hline
4&--&--&0.04584&0.22049\\\hline
5&--&0.06611&0.23887&0.28137\\\hline
6&--&0.15811&0.28980&0.30063\\\hline
7&--&0.24269&0.28562&0.29235\\\hline
8&--&0.26988&0.28517&0.28715\\\hline
9&0.12444&0.29046&0.29616&0.29759\\\hline
10&0.17857&0.29522&0.30876&0.30886\\\hline
\end{tabular}
\end{minipage}
\begin{minipage}{3.5in}
\centering
\caption{$\ceff$ for eighth order methods} \label{tab:8ord}
\begin{tabular}{|l|l|l|l|l|}
\hline
\multicolumn{5}{|c|}{$\ceff$ for eighth order methods} \\ \hline
$s\backslash k$ &2&3&4&5\\\hline
2&--&--&--&--\\\hline
3&--&--&--&--\\\hline
4&--&--&--&--\\\hline
5&--&--&0.04781&0.10007\\\hline
6&--&--&0.07991&0.22574\\\hline
7&--&--&{\color{gray} 0.14818}&{\color{gray} 0.22229}\\\hline
8&--&{\color{gray} 0.09992}&{\color{gray} 0.16323}&{\color{gray} 0.19538}\\\hline
9&--&0.14948&0.21012&0.23826\\\hline
10&--&0.20012&0.21517&0.24719\\\hline
\end{tabular}
\end{minipage}
\end{table}

\vspace{-.25in}

\subsubsection{Eighth order methods} Explicit eighth order two-step RK methods found in
\cite{tsrk} require at least 11 stages and have $\ceff\le 0.078$.
Much larger values of $\ceff$ can be achieved with fewer stages by using additional 
steps, as shown in Figure \ref{fig:8thO} and Table \ref{tab:8ord}.  The best method
has $\ceff \approx 0.247$; to achieve the same efficiency with a linear multistep method
requires the use of more than thirty steps~\cite{ketcheson2009a}.
Once again, due to the number of coefficients and constraints this was a difficult optimization problem and
we were not able to converge to the best methods from a "cold start". This is evident in our $(s,k)=(7,4),
(7,5), (8,3), (8,4), (8,5)$ methods which have a smaller SSP coefficient than those in 
\cite{VaillancourtJSC2012,VaillancourtJCAM2014}.
However, converting the methods in \cite{nguyenthesis} to our form we were able 
to replicate their results while tightening the optimizer parameters {\tt TolCon, TolFun} 
and {\tt TolX} from $10^{-12}$ in their work to $10^{-14}$.

\subsubsection{Ninth order methods} 
Explicit two-step RK methods with positive SSP coefficient and order nine cannot exist~\cite{tsrk}.
For orders higher than eight, finding practical multistep or Runge--Kutta methods
is a challenge even when the SSP property is not required. Numerical optimization of 
such high order MSRK methods is computationally intensive, so we have restricted our
search to a few combinations of stage and step number. We are able to break the order barrier
of the two step methods by finding a $(s,k)=(10,3)$ method. 
Investigating methods with four steps, we obtain
a $(s,k)= (8,4)$ method with $\ceff=0.1276$, and a  $(s,k)= (9,4)$ method with $\ceff = 0.1766$.
We also found a $(9,5)$ method with $\ceff=0.1883$. By comparison, a multistep method requires
23 steps for $\ceff=0.116$ and 28 steps for $\ceff=0.175$. 
These methods also compare favorably to the $(s,k)=(8,5)$
method in \cite{nguyenthesis}, that has $\ceff=0.153$.  However, based on our experience with eighth order
methods we do not claim that our new methods are optimal.

\subsubsection{Tenth order methods} The search for tenth order methods is complicated by the large
number of constraints and the large number of steps and stages required, and so we did not pursue
optimization of these methods in general. However, we obtained an $(s,k)= (20,3)$ with $\ceff= 0.0917$ 
which demonstrates that  3-step methods with order 10  exist. We also obtained a $(k,s)=(8,6)$ method
with $\ceff=0.839$. Once again, we do not claim that this methods are optimal.
While these method  have  small effective coefficients, they demonstrate that it is possible to find
tenth order SSP methods with much less than the  22 steps required for linear multistep
methods. For comparison, the optimal multistep method with 22 steps and order 10 has
$\ceff=0.10$ \cite{SSPbook2011}.

\section{Numerical Results\label{sec:test}}
\vspace{-.15in}
In this section we present numerical tests of the optimized MSRK methods identified above.
The numerical tests have three purposes: (1) to verify that the methods have the designed order of accuracy;
(2) to demonstrate the value of high order time-stepping methods when using high-order spatial discretizations;
and (3) to study the strong stability properties of the newly designed MSRK methods in practice, on test cases
for which the forward Euler method is known to be total variation diminishing or 
positivity preserving.
The scripts for many of these tests can be found at \cite{testsuite}.

\vspace{-.15in}

\subsection{Order Verification}
\vspace{-.1in}

Convergence studies for ordinary differential equations were performed using the  van der Pol 
oscillator, a nonlinear system, to  confirm the design orders of the methods. As these methods 
were designed for use as time integrators for partial differential equations, we include a
convergence study for a PDE with high order spatial discretization.



\begin{figure}[b!]
\subfigure[van der Pol]{\includegraphics[scale=.35]{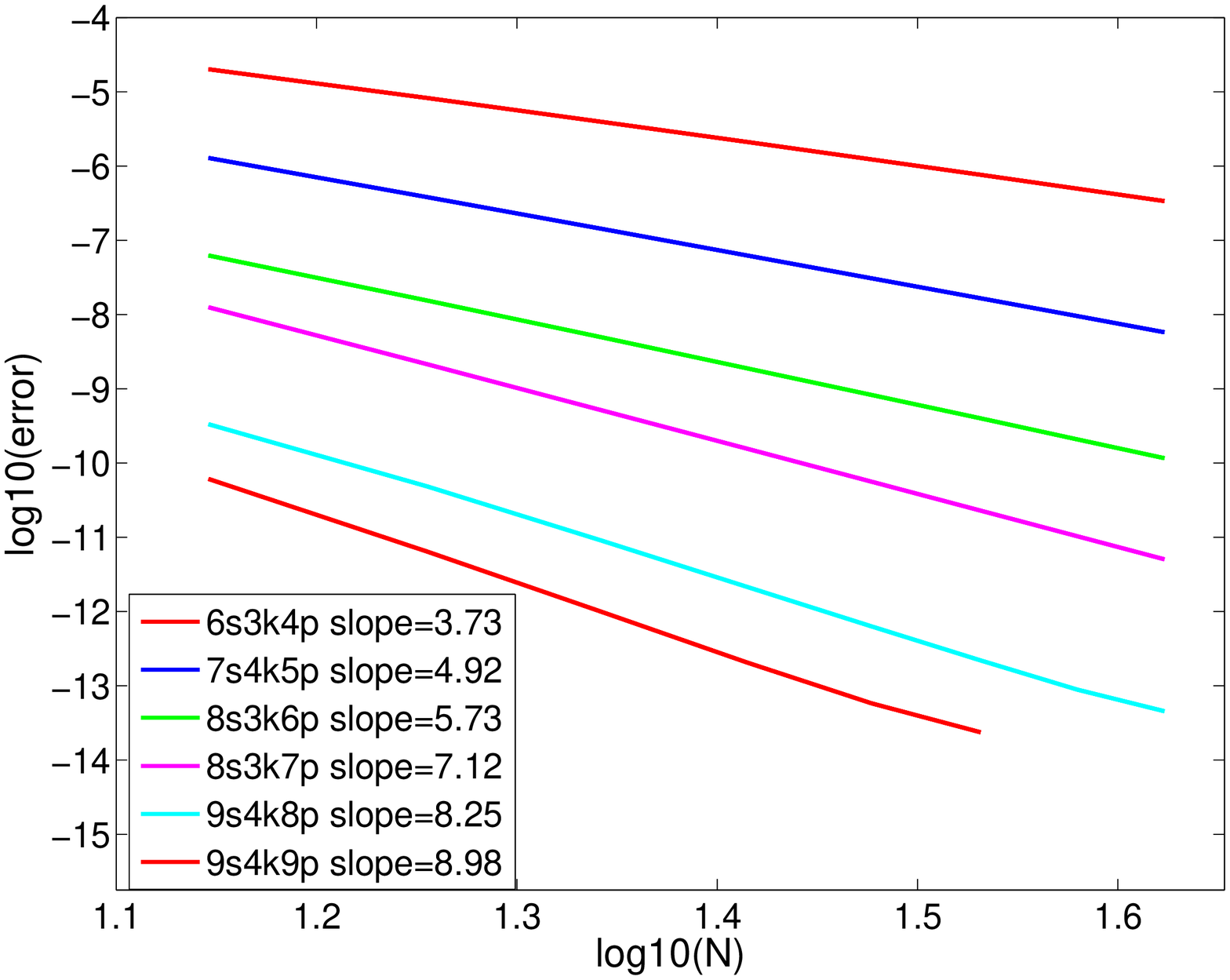}}
\subfigure[linear advection\label{fig:conv-adv}]{\includegraphics[scale=.382]{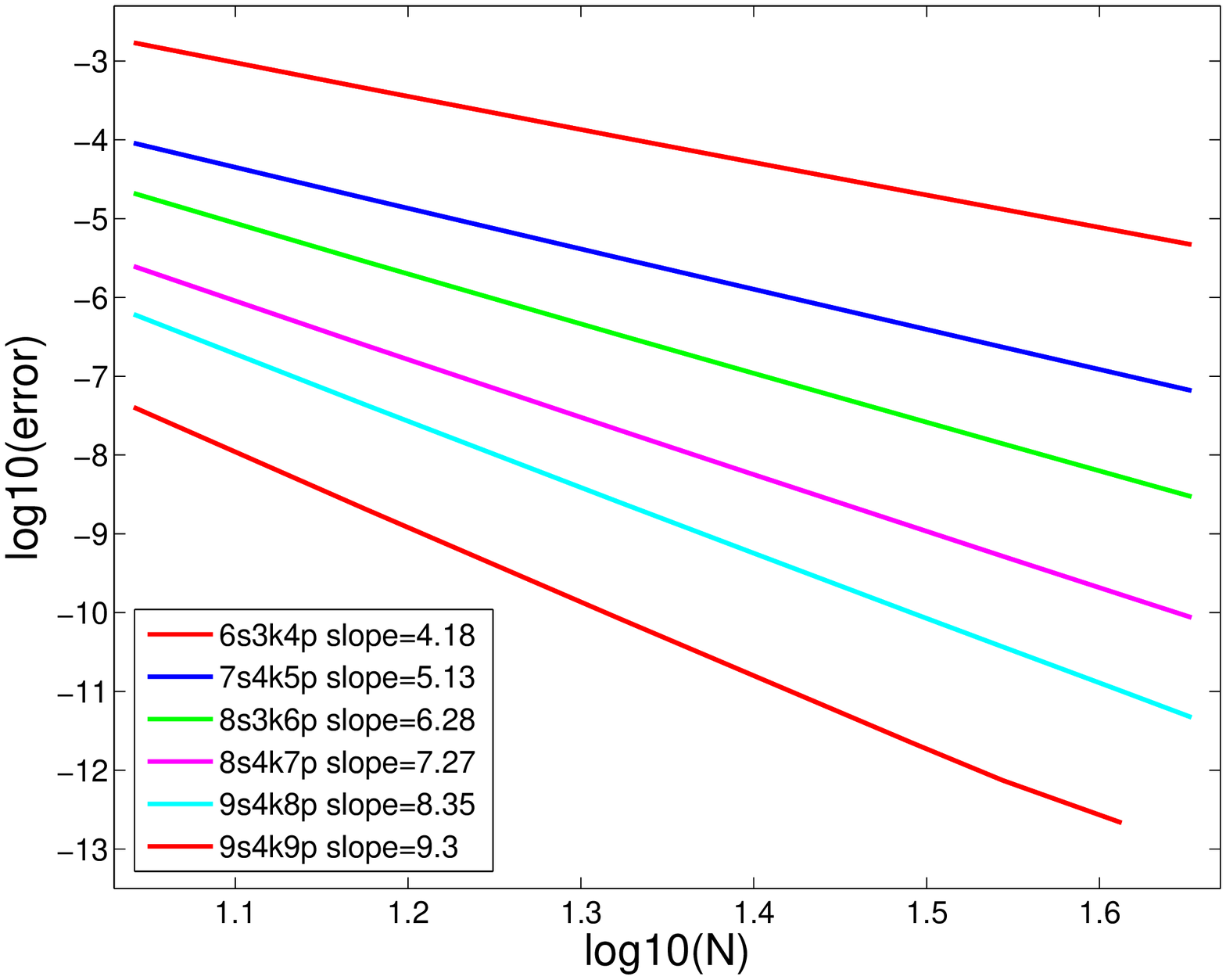}}
\caption{Order verification of multistep Runge--Kutta methods on ordinary 
differential equations (left) and partial differential equation (right). } \label{fig:conv}
\end{figure}   

\noindent{\bf The van der Pol oscillator problem.}  
The van der Pol problem is:
\begin{eqnarray}
&u_1' = u_2 \\ 
&u_2' = \frac{1}{\epsilon} (-u_1 + (1-u_1^2) u_2)
\end{eqnarray}
We use $\epsilon = 10$ and initial conditions $u_0 = (0.5;  0)$. 
This was run to final time $T_{final} =  4.0$, with $\Delta t = \frac{T_{final}}{N-1} $ where 
$N=15,19,23,27,31,35,39,43$.
Starting values and exact solution (for error calculation)  were calculated by  the highly 
accurate MATLAB ODE45 routine with tolerances set to $10^{-14}$. 
 In Figure \ref{fig:conv}(a)  we show the convergence of the selected  
 $(s,k,p)=(6,3,4), (7,4,5), (8,3,6), (8,3,7), (9,4,8), $ and $(9,4,9)$ methods
for $u_1$.
Due to space limitations, we present only the results for a few methods, 
one of each order up to $p=9$. The new multistep Runge--Kutta  methods exhibit the correct order of accuracy.


\noindent{\bf Linear advection with a Fourier spectral method.} 
For the PDE convergence test, we chose the Fourier spectral method on the advection equation with sine wave  initial conditions and 
periodic boundaries:    \vspace{-.2in}
\begin{eqnarray}
u_t & = & - u_x  \; \; \; x \in [0,1]  \\ \nonumber
u(0,x)& =& \sin(4\pi x) \; \; \; \; u(t,0)=u(t,1)  \nonumber
\end{eqnarray}
The exact solution to this problem is a sine wave with period $4$ that travels in time. Due to the periodicity of the exact solution,
the Fourier spectral method gives us an exact solution in space \cite{HGG2007} once we have two points per wavelength, allowing 
us to isolate the effect of the  temporal discretization on the error. We run this problem with $N=(11,15,21,25,31,35,41,45)$ 
to $T_{final}=1$ with $\Delta t = 0.4 \Delta x$, where $\Delta x = \frac{1}{N-1}$. 
For each multi-step Runge--Kutta method of order $p$ we generated the  $k-1$ initial values using 
the third order Shu-Osher SSP Runge--Kutta method with
a very small time-step  $\Delta t^{p/3}$. 
Errors are computed at the final time, compared to the exact solution. 
Figure \ref{fig:conv-adv} contains the $l_2$ norm of the errors,
and demonstrates that the methods achieved the expected convergence rates.

   \vspace{-.25in}
\subsection{Benefits of high order time discretizations}
   \vspace{-.1in}
\begin{wrapfigure}[17]{r}{3.6in} \vspace{-.27in}
\includegraphics[scale=.5]{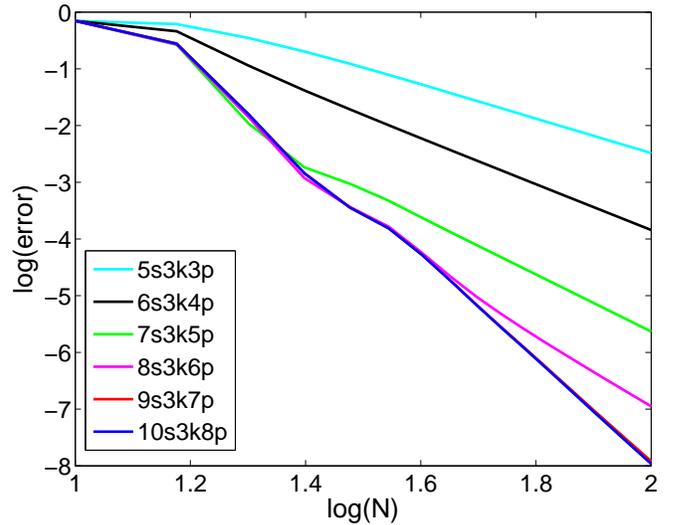} 
\caption{\small Convergence of a 2D advection equation with $9$th order WENO in space and 
MSRK in time.}
 \label{WENO9}
\end{wrapfigure}
High-order spatial discretizations for hyperbolic PDEs have usually been paired with 
lower-order time 
discretizations; e.g. \cite{carrillo2003,cheng2003,cheruvu2007,enright2002,feng2004,jin2005,
labrunie2004,peng1999,tanguay2003}. 
Although spatial truncation errors are often observed to be larger than temporal errors in practice,
this discrepancy can lead to loss of accuracy 
unless the time-step is
significantly reduced. If the order of the time-stepping method is $p_1$ 
and the order of the spatial method is $p_2$ then asymptotic convergence at rate $p_2$ is assured only if
$\Delta t = {\mathcal O}(\Delta x^{p_2/p_1})$.  For hyperbolic PDEs, one typically wishes to take 
$\Delta t = {\mathcal O}(\Delta x)$ for accuracy reasons.


In the following example we solve the two-dimensional advection equation
\begin{align*}
    u_t + u_x + u_y & = 0
\end{align*}
over the unit square with periodic boundary conditions in each direction
and initial data $u(0,x) = \sin(2\pi( x+y))$. We take $\Delta x = \Delta y = \frac{1}{N-1}$. 
We solve for $0\le t \le \frac{1}{8}$ 
with $\Delta t = \frac{1}{4} \Delta x$. We use ninth-order WENO finite differences in space. 
For each multi-step Runge--Kutta method of order $p$ we generated the  $k-1$ initial values using 
the third order Shu-Osher SSP Runge--Kutta method with
a very small time-step  $\Delta t^{p/3}$.
Figure \ref{WENO9} shows the accuracy of several of our high order multistep Runge--Kutta methods 
applied to this problem.
Observe  that  while methods of order $p \leq 6$ exhibit an asymptotic convergence rate of
less than 9th order,
our newly found methods of order $p\geq 7$ allow the high order behavior of the WENO to become apparent.

\vspace{.02in}
\subsection{Strong stability performance of the new MSRK methods}
In this section we discuss the strong stability performance of the new methods in practice. 
The SSP condition is a very general condition: it holds for any convex functional and any 
starting value, for arbitrary nonlinear non-autonomous equations, assuming only that 
the forward Euler method satisfies the corresponding monotonicity condition. 
In other words, it is a bound based on the worst-case behavior.  Hence it should not be
surprising that larger step sizes are possible when one considers a particular problem
and a particular convex functional.

Here we explore the behavior of these methods in practice on 
the linear advection and nonlinear  Buckley-Leverett equations,
looking only at the total variation and positivity properties.  The scripts for
these tests can be found at \cite{testsuite}.


\noindent{\bf Example 1: Advection.} 
Our first example is the advection equation with a step function initial condition:
\begin{align}
u_t + u_x & = 0 \hspace{.75in}
    u(0,x)  =
\begin{cases}
1, & \text{if } 0 \leq x \leq 1/2 \\
0, & \text{if } x>1/2 \nonumber
\end{cases}
\end{align}
on the domain $[0,1)$ with periodic boundary conditions.
The problem was semi-discretized using a first-order forward difference on a grid 
with $N=101$ points and evolved to a final time of $t=\frac{1}{8}$.  
We used the exact  solution for the  $k-1$ initial values. 
Euler's method is TVD and positive for step sizes up to $\DtFE = \Dx$.
Table \ref{tab:linadvTVD} shows the normalized observed time step   $\frac{\Delta t_{TVD}}{\Delta x}$ for which each
method maintains the total variation diminishing property
and the observed time step  $\frac{\Delta t^+}{\Delta x}$ for which each method maintains positivity.
We compare these values to the normalized  time-step guaranteed by the theory,
$ \sspcoef \frac{\Delta t_{FE}}{\Delta x}$. 
The table also compares the effective observed TVD time-step $\frac{1}{s} \frac{\Delta t_{TVD}}{\Delta x}$,
and the effective positivity time step $\frac{1}{s} \frac{\Delta t^+}{\Delta x}$,
with the effective time-step given by the theory $\ceff \frac{\Delta t_{FE}}{\Delta x}$. 
These examples confirm that the observed positivity preserving time-step correlates well with the size of the SSP coefficient, 
and these methods compare favorably with the baseline methods. Also,  the methods perform in 
practice as well or better than the lower bound guaranteed by the theory. 

\begin{table}
{\small
\begin{tabular}{|c|cc|cc|cc|} \hline
method & $ \frac{\Delta t_{TVD}}{\Delta x} $ & $ \frac{1}{s}  \frac{\Delta t_{TVD}}{\Delta x} $  
&   $ \sspcoef \frac{\Delta t_{FE}}{\Delta x} $ & $\ceff \frac{\Delta t_{FE}}{\Delta x}$   
& $ \frac{\Delta t^+}{\Delta x} $ & $ \frac{1}{s}  \frac{\Delta t^+}{\Delta x} $ \\ \hline
SSPRK 3,3 & 1.000 & 0.333 & 1.000 & 0.333 & 1.028 & 0.342    \\
 (2,3,3) &  1.113 &  0.556 &  1.113 &  0.556 &  1.113 &  0.556  \\
 (6,3,3) &  3.777 &  0.629 &  3.777 &  0.629 &  3.777 &  0.629  \\
 (7,3,3) &  6.300 &  0.900 &  4.484 &  0.641 &  6.300 &  0.900  \\ \hline
 (2,3,4) &  0.495 &  0.248 &  0.495 &  0.248 &  0.495 &  0.248  \\
 (3,4,4) &  1.365 &  0.455 &  1.365 &  0.455 &  1.365 &  0.455  \\
non-SSP RK4,4 &  1.000 &  0.250 &  0.000 &  0.000 &  1.031 &  0.258  \\
SSP RK10,4 & 6.00 & 0.600 & 6.000 & 0.600 & 6.032 & 0.603    \\
 (7,3,4) &  3.749 &  0.536 &  3.749 &  0.536 &  4.001 &  0.572  \\ \hline
 (3,3,5) &  0.641 &  0.214 &  0.638 &  0.213 &  0.663 &  0.221  \\
 (3,4,5) &  1.001 &  0.334 &  1.001 &  0.334 &  1.001 &  0.334  \\
 (3,5,5) &  1.162 &  0.387 &  1.162 &  0.387 &  1.162 &  0.387  \\
 (6,3,5) &  2.423 &  0.404 &  2.423 &  0.404 &  2.423 &  0.404  \\ \hline
 (3,5,6) &  0.657 &  0.219 &  0.657 &  0.219 &  0.657 &  0.219  \\
 (4,4,6) &  0.985 &  0.246 &  0.971 &  0.243 &  0.985 &  0.246  \\
 (5,3,6) &  1.361 &  0.272 &  1.361 &  0.272 &  1.361 &  0.272  \\
 (6,5,6) &  2.067 &  0.345 &  2.067 &  0.345 &  2.139 &  0.357  \\
 (9,3,6) &  3.146 &  0.350 &  3.144 &  0.349 &  3.578 &  0.398  \\ \hline
 (4,5,7) &  0.901 &  0.225 &  0.882 &  0.220 &  0.917 &  0.229  \\
 (7,3,7) &  1.699 &  0.243 &  1.699 &  0.243 &  1.699 &  0.243  \\
 (7,4,7) &  1.999 &  0.286 &  1.999 &  0.286 &  1.999 &  0.286  \\ \hline 
 (8,3,8) &  0.898 &  0.112 &  0.799 &  0.100 &  0.898 &  0.112  \\
 (9,5,8) &  2.058 &  0.229 &  2.058 &  0.229 &  2.222 &  0.247  \\ \hline
 (9,4,9) &  1.638 &  0.182 &  1.590 &  0.177 &  1.672 &  0.186  \\ \hline
(20,3,10) &  2.146 &  0.107 &  1.835 &  0.092 &  2.209 &  0.110  \\ \hline
  \end{tabular}
\caption{Observed total variation diminishing (TVD) and positivity  time-step and 
        effective TVD and positivity time-step (normalized by the spatial step)
        compared   with the theoretical values for Example 1.}
         \label{tab:linadvTVD}}
  \end{table}

\noindent{\bf Example 2: Buckley-Leverett Problem:} 
We solve the Buckley-Leverett equation, a nonlinear PDE used to model two-phase flow through porous media:
\begin{align*}
    u_t+f(u)_x & = 0, & \text{ where } f(u) = \frac{u^2}{u^2 +a(1-u)^2},
\end{align*}
on $x\in[0,1)$, with periodic boundary conditions. 
We take $a=\frac{1}{3}$ and initial condition
\begin{eqnarray}
    u(x,0) =
\begin{cases}
1/2, & \text{if }x\ge1/2 \\
0, & \text{otherwise.}
\end{cases}
\end{eqnarray}
The problem is semi-discretized using a conservative scheme
with a Koren Limiter as in \cite{tsrk} with $ \Delta x = \frac{1}{100}$, and run to $t_f = \frac{1}{8}$. 
 For this problem the theoretical TVD time-step is $\Delta t_{FE} = \frac{1}{4} \Delta x = 0.0025$. 
 For each multi-step Runge--Kutta method of order $p$ we generated the  $k-1$ initial values using 
the third order Shu-Osher SSP Runge--Kutta method with
a very small time-step  $\Delta t^{p/3}$.
 \begin{figure}[H]
 \subfigure[Third order methods]{\includegraphics[scale=.32]{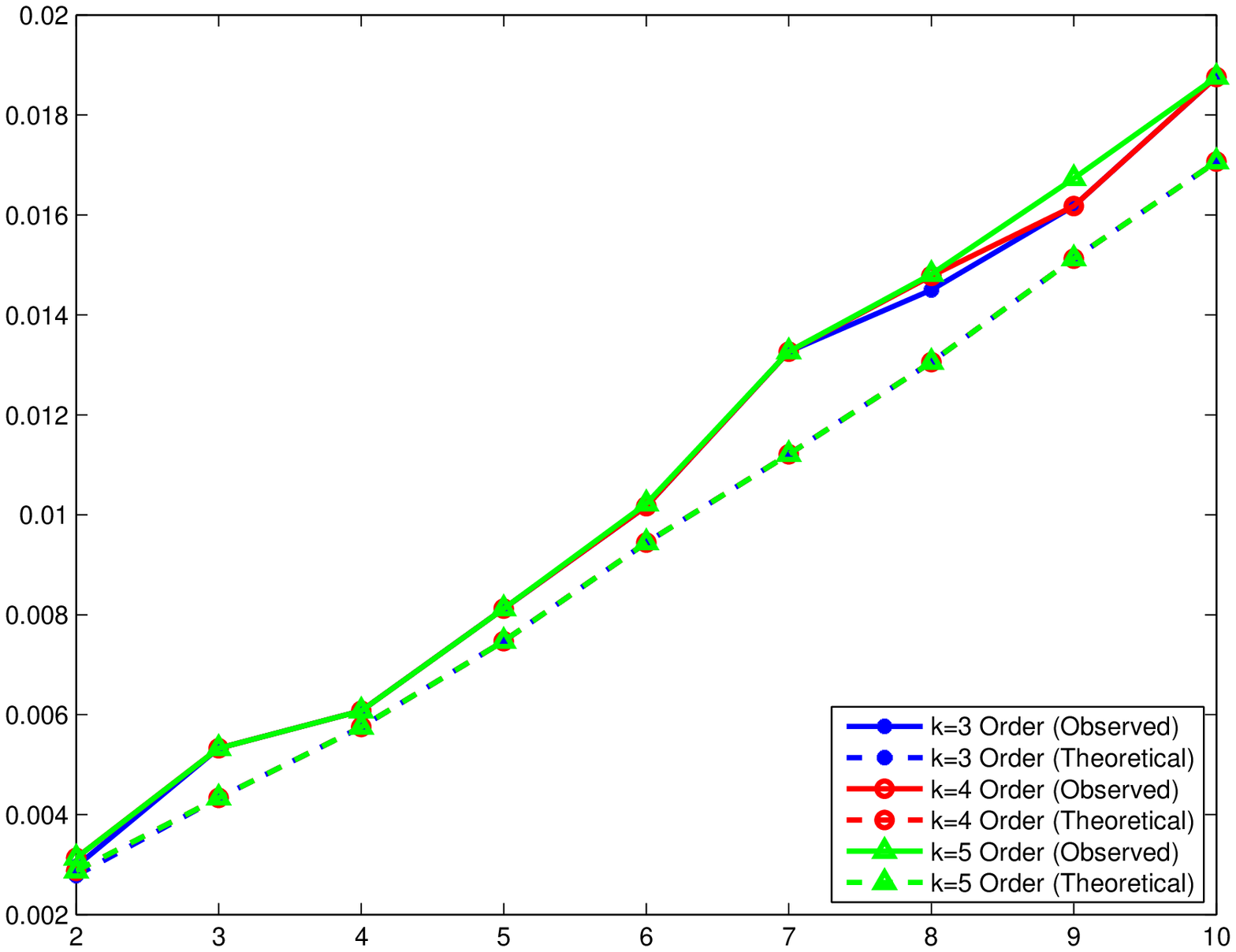}}
 \subfigure[Fourth order methods]{\includegraphics[scale=.32]{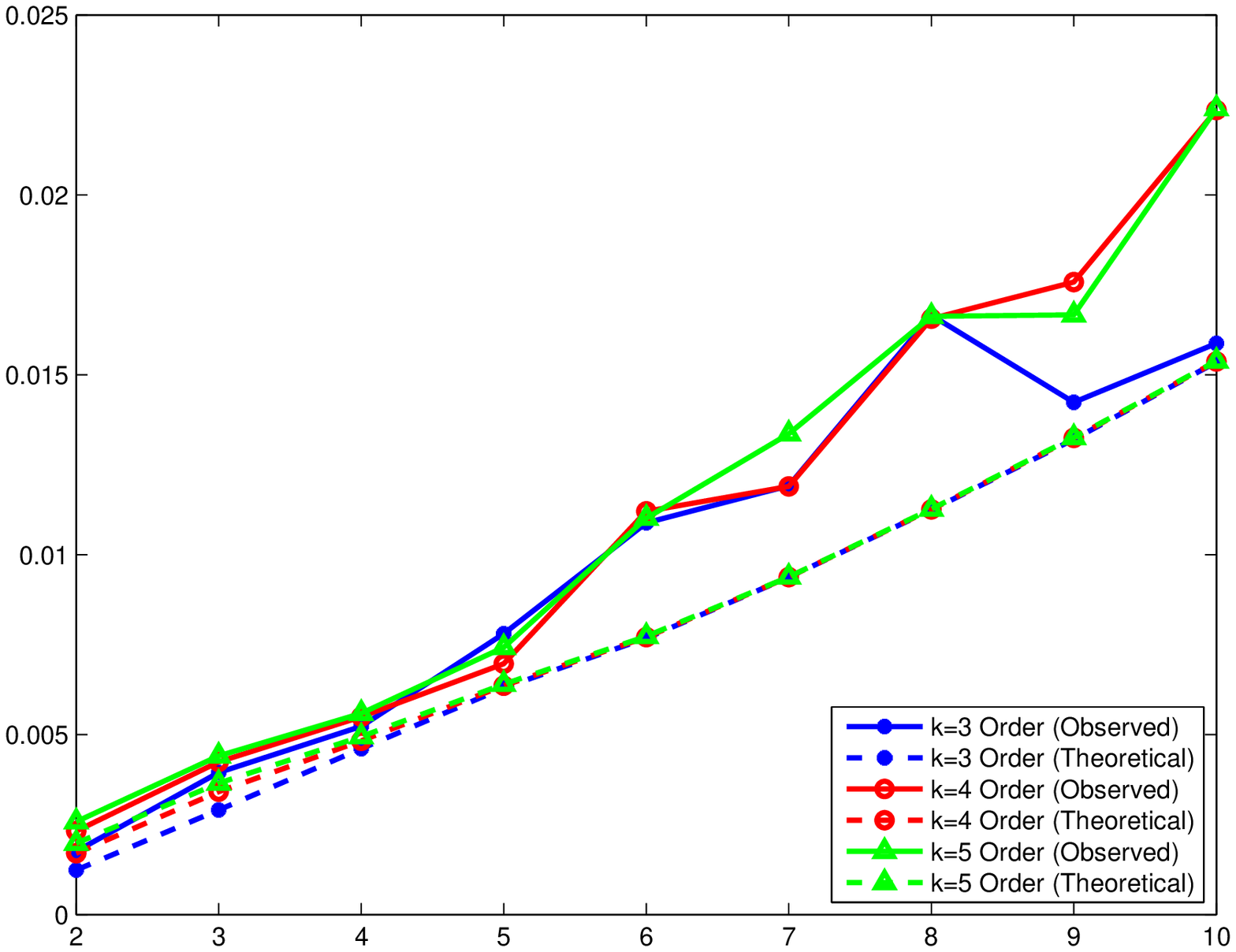}}
 \subfigure[Fifth order methods]{\includegraphics[scale=.32]{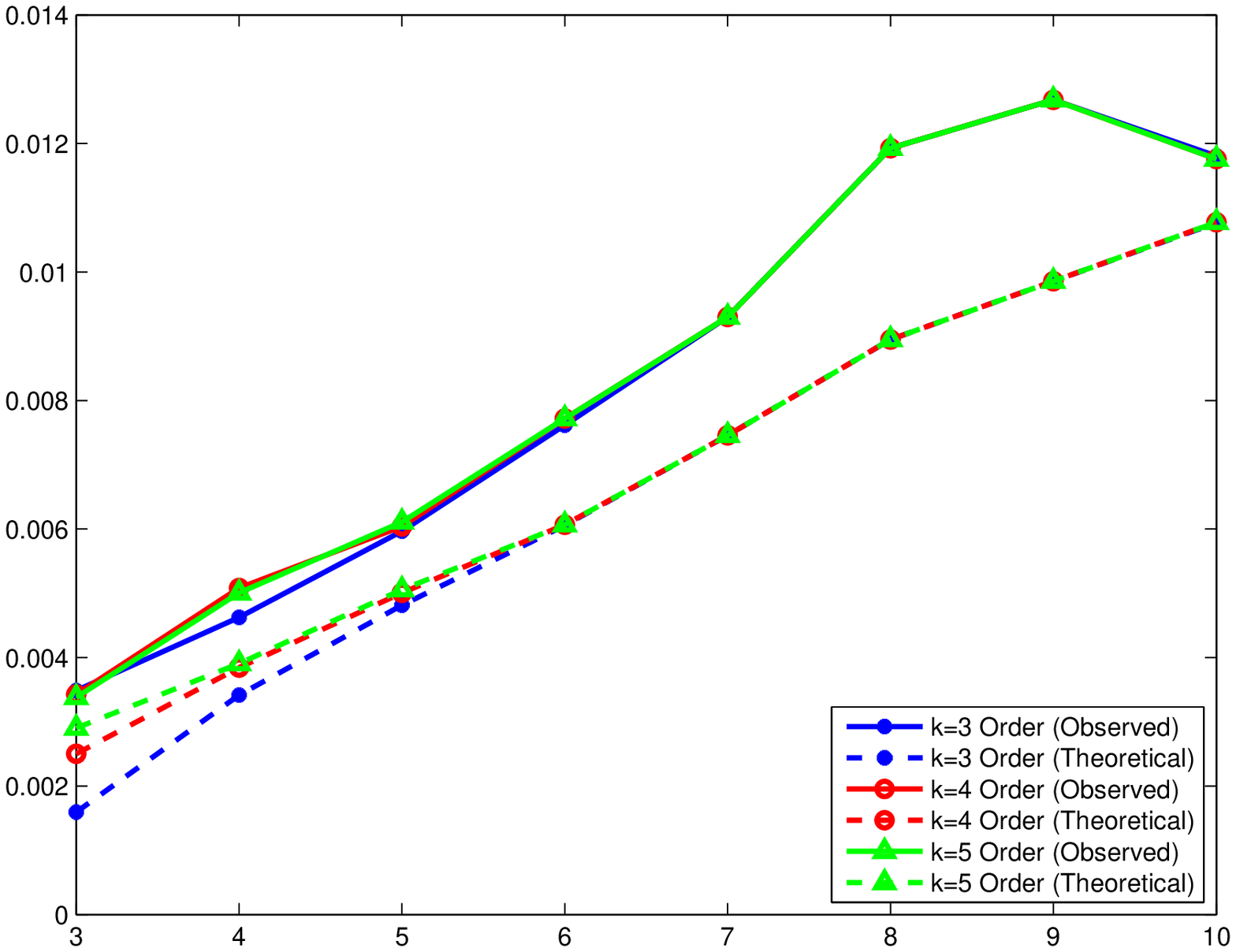}} \\
 \subfigure[Sixth order methods]{\includegraphics[scale=.32]{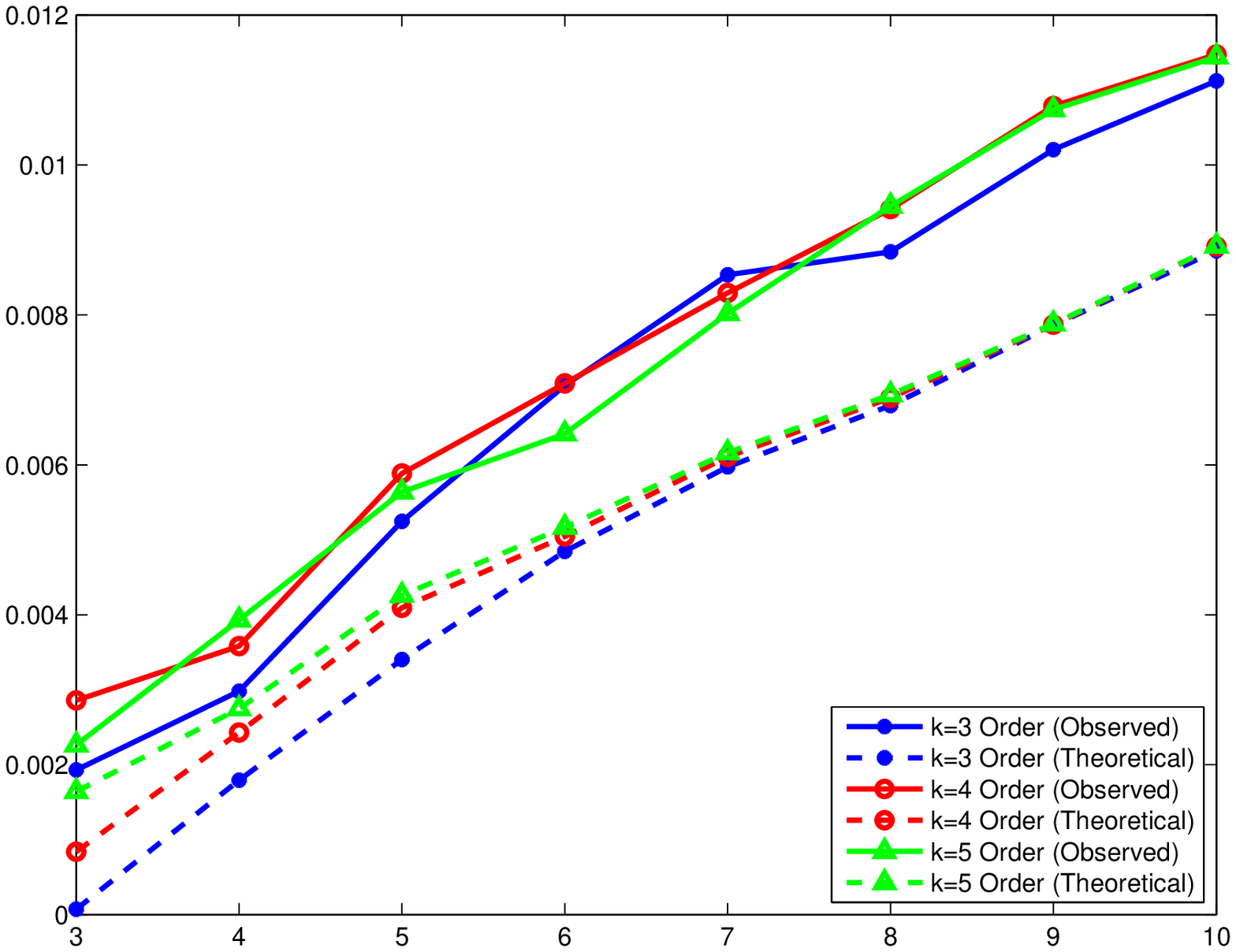}}
 \subfigure[Seventh order methods]{\includegraphics[scale=.32]{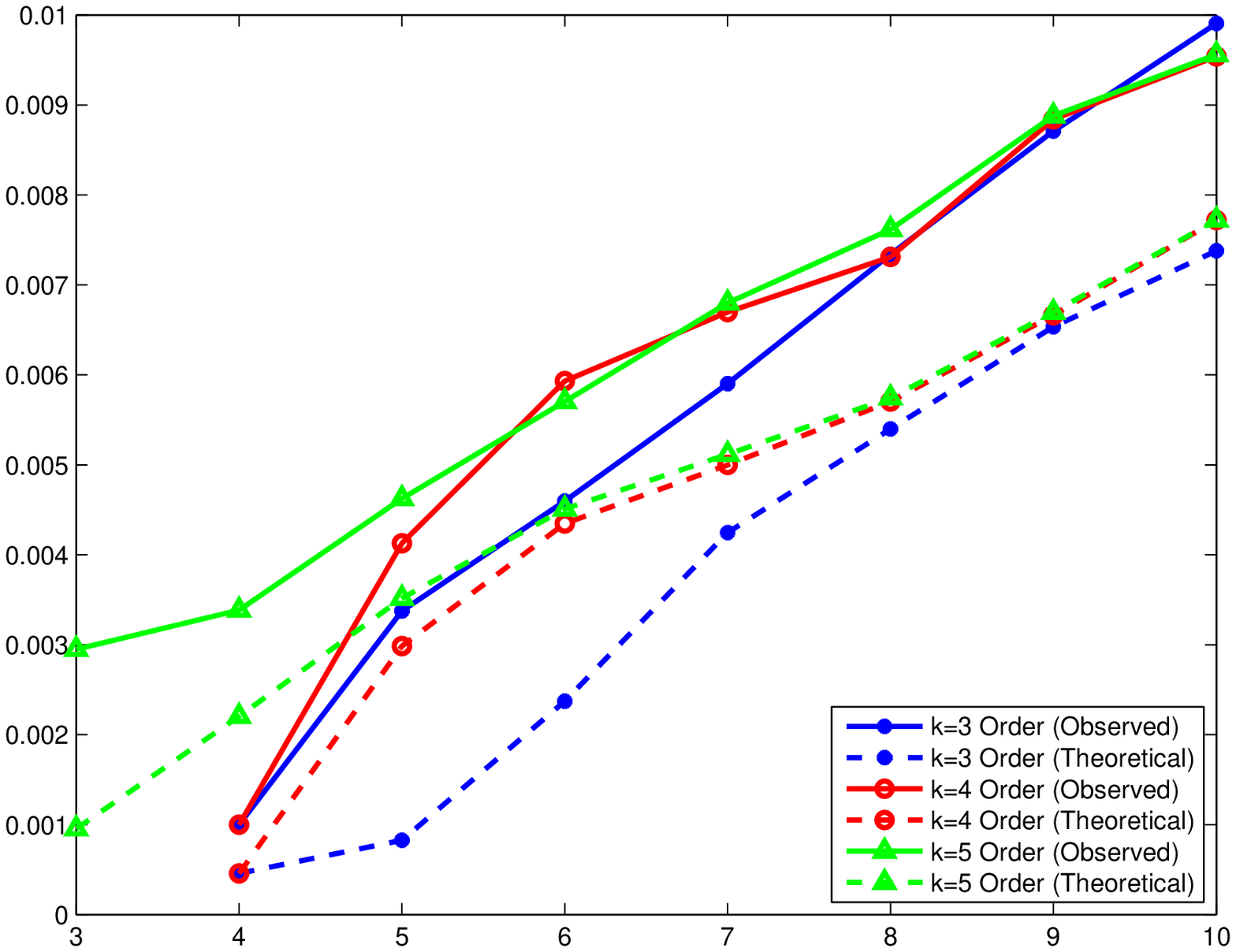}}
 \subfigure[Eighth order methods]{\includegraphics[scale=.32]{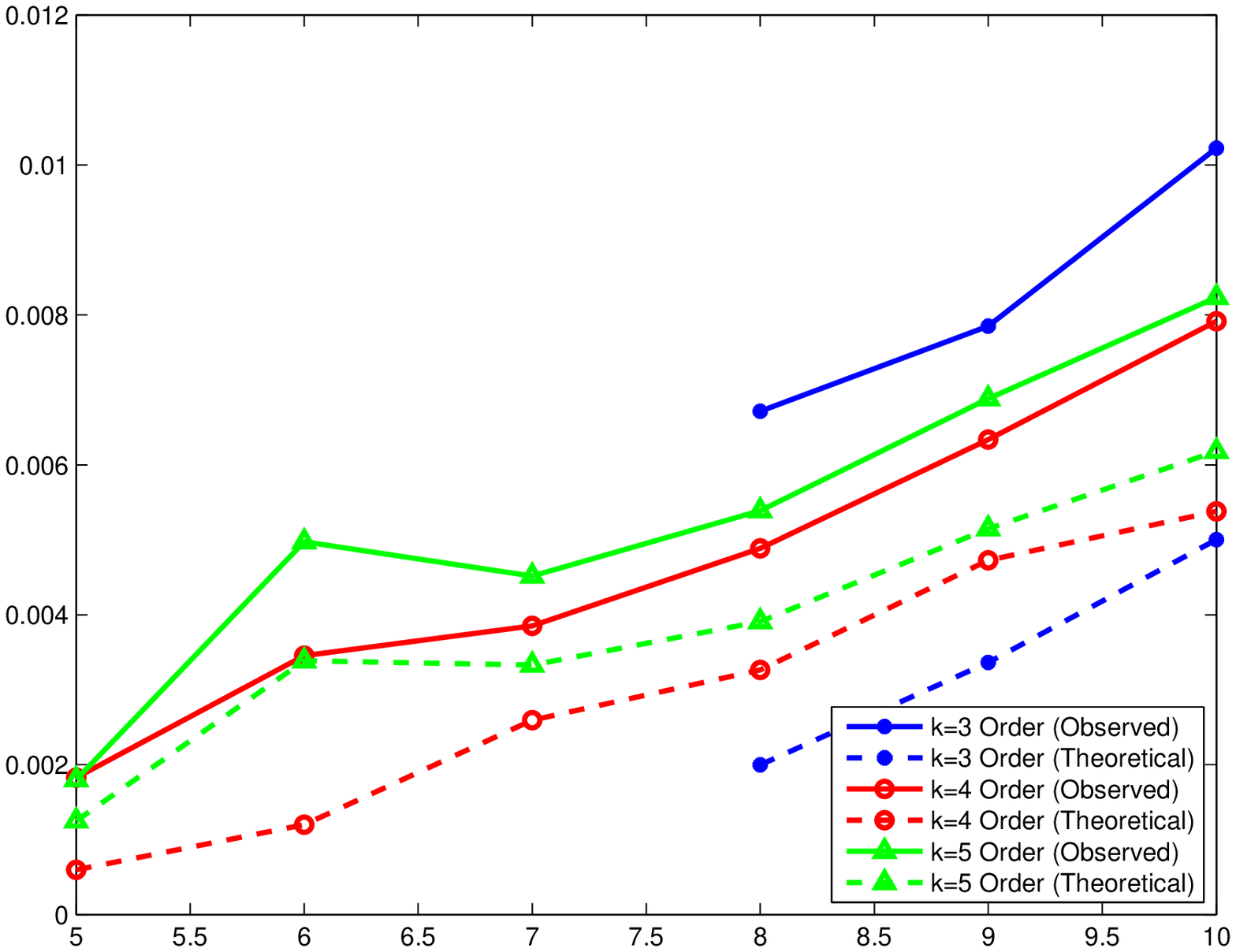}}
\caption{ The observed normalized TVD time-step ($\Delta t/\Delta x$) compared to the theoretical 
for multistep Runge--Kutta methods of order $p=3, . . ., 8$ using Example 2.}
\label{fig:BLtvd}
\end{figure}

 \begin{wraptable}[15]{r}{2.5in} 
{\small
\begin{tabular}{|c|cc|} \hline 
method & $\frac{\Delta t^{TVD}}{4 \Delta x}$  & $\frac{\Delta t^+}{4 \Delta x}$ \\ \hline
 (3,1,3) &  0.114 &  0.226    \\ 
 (3,3,3) &  0.133 &  0.235    \\ 
 (6,3,3) &  0.254 &  0.621    \\ \hline
 (4,1,4) &  0.185 &  0.302    \\ 
(10,1,4) &  0.419 &  0.754    \\ 
 (2,3,4) &  0.044 &  0.089    \\ 
 (3,4,4) &  0.106 &  0.179    \\ 
 (7,3,4) &  0.298 &  0.565    \\ \hline
 (3,3,5) &  0.087 &  0.175    \\ 
 (6,3,5) &  0.190 &  0.375    \\ 
 (3,4,5) &  0.086 &  0.173    \\ 
 (3,5,5) &  0.085 &  0.195    \\  \hline
 (5,3,6) &  0.131 &  0.262    \\ 
 (9,3,6) &  0.255 &  0.501    \\ 
 (4,4,6) &  0.090 &  0.191    \\ 
 (3,5,6) &  0.057 &  0.117    \\ 
 (6,5,6) &  0.160 &  0.349    \\  \hline
 (7,3,7) &  0.148 &  0.286    \\ 
 (8,3,7) &  0.183 &  0.346    \\ 
 (7,4,7) &  0.167 &  0.353    \\ 
 (4,5,7) &  0.085 &  0.172    \\ \hline
 (8,3,8) &  0.168 &  0.353    \\ 
 (6,4,8) &  0.086 &  0.195    \\ 
 (6,5,8) &  0.124 &  0.256    \\ 
 (9,5,8) &  0.172 &  0.348    \\ \hline
 (9,4,9) &  0.166 &  0.360    \\  \hline
(20,3,10) &  0.356 &  0.630    \\ 
\end{tabular}
\caption{Observed positivity preserving normalized time-step   compared 
  with the TVD normalized time-step for Example 2.}
 \label{tab:BLpos} }
 \end{wraptable}
The plots in Figure \ref{fig:BLtvd} show the observed  normalized time-step for TVD ($\Delta t/\Delta x$) for the number of stages, for each family of $k$-step methods. 
The dotted lines are the corresponding theoretical TVD time-step for these methods. We see that the observed values
are  significantly  higher than the theoretical values, but the observed values generally increase with the number of stages as
predicted. In Table  \ref{tab:BLpos} we compare the positivity preserving time-step to the TVD time-step. We note that the TVD time step
is always smaller than the positivity time-step, demonstrating the dependence of the observed time-step on the particular property 
desired.

\vspace{4.5in}
{\bf Acknowledgment.} This publication is based on work supported by Award No.
FIC/2010/05 - 2000000231, made by King Abdullah University of Science and
Technology (KAUST) and on AFOSR grant FA-9550-12-1-0224.

%
%
%

\newpage
\bibliography{msrk}

\begin{thebibliography}{10}

\bibitem{albrecht1996}
{\sc P.~Albrecht}, {\em The {R}unge--{K}utta theory in a nutshell}, SIAM
  Journal on Numerical Analysis, 33 (1996), pp.~1712--1735.

\bibitem{carrillo2003}
{\sc J.~Carrillo, I.~M. Gamba, A.~Majorana, and C.-W. Shu}, {\em A
  {WENO}-solver for the transients of {B}oltzmann--{P}oisson system for
  semiconductor devices: performance and comparisons with {M}onte {C}arlo
  methods}, Journal of Computational Physics, 184 (2003), pp.~498--525.

\bibitem{cheng2003}
{\sc L.-T. Cheng, H.~Liu, and S.~Osher}, {\em Computational high-frequency wave
  propagation using the level set method, with applications to the
  semi-classical limit of {S}chr{\"o}dinger equations}, Comm. Math. Sci., 1
  (2003), pp.~593--621.

\bibitem{cheruvu2007}
{\sc V.~Cheruvu, R.~D. Nair, and H.~M. Turfo}, {\em A spectral finite volume
  transport scheme on the cubed-sphere}, Applied Numerical Mathematics, 57
  (2007), pp.~1021--1032.

\bibitem{constantinescu2009}
{\sc E.~Constantinescu and A.~Sandu}, {\em {Optimal explicit
  strong-stability-preserving general linear methods}}, SIAM Journal on
  Scientific Computing, 32 (2009), pp.~3130--3150.

\bibitem{enright2002}
{\sc D.~Enright, R.~Fedkiw, J.~Ferziger, and I.~Mitchell}, {\em A hybrid
  particle level set method for improved interface capturing}, Journal of
  Computational Physics, 183 (2002), pp.~83--116.

\bibitem{feng2004}
{\sc L.~Feng, C.~Shu, and M.~Zhang}, {\em A hybrid cosmological
  hydrodynamic/{$N$-}body code based on a weighted essentially nonoscillatory
  scheme}, The Astrophysical Journal, 612 (2004), pp.~1--13.

\bibitem{ferracina2004}
{\sc L.~Ferracina and M.~N. Spijker}, {\em Stepsize restrictions for the
  total-variation-diminishing property in general {R}unge--{K}utta methods},
  SIAM Journal of Numerical Analysis, 42 (2004), pp.~1073--1093.

\bibitem{ferracina2005}
\leavevmode\vrule height 2pt depth -1.6pt width 23pt, {\em An extension and
  analysis of the {S}hu--{O}sher representation of {R}unge--{K}utta methods},
  Mathematics of Computation, 249 (2005), pp.~201--219.

\bibitem{testsuite}
{\sc S.~Gottlieb and D.~Higgs}, {\em Strong stability preserving tools test
  suite}.
\newblock \url{http://sspsite.org/ssp_tools/}.

\bibitem{sspsite}
{\sc S.~Gottlieb, D.~Higgs, and D.~I. Ketcheson}, {\em Strong stability
  preserving site}.
\newblock \url{http:www.sspsite.org/msrk.html}.

\bibitem{gottlieb2009}
{\sc S.~Gottlieb, D.~I. Ketcheson, and C.-W. Shu}, {\em {High Order Strong
  Stability Preserving Time Discretizations}}, Journal of Scientific Computing,
  38 (2009), pp.~251--289.

\bibitem{SSPbook2011}
\leavevmode\vrule height 2pt depth -1.6pt width 23pt, {\em Strong Stability
  Preserving Runge--Kutta and Multistep Time Discretizations}, World Scientific
  Press, 2011.

\bibitem{gottlieb2001}
{\sc S.~Gottlieb, C.-W. Shu, and E.~Tadmor}, {\em {Strong Stability Preserving
  High-Order Time Discretization Methods}}, SIAM Review, 43 (2001),
  pp.~89--112.

\bibitem{HGG2007}
{\sc J.~Hesthaven, S.~Gottlieb, and D.~Gottlieb}, {\em Spectral methods for
  time dependent problems}, Cambridge Monographs of Applied and Computational
  Mathematics, Cambridge University Press, 2007.

\bibitem{higueras2004a}
{\sc I.~Higueras}, {\em On strong stability preserving time discretization
  methods}, Journal of Scientific Computing, 21 (2004), pp.~193--223.

\bibitem{higueras2005a}
\leavevmode\vrule height 2pt depth -1.6pt width 23pt, {\em Representations of
  {R}unge--{K}utta methods and strong stability preserving methods}, SIAM
  Journal On Numerical Analysis, 43 (2005), pp.~924--948.

\bibitem{huang2009}
{\sc C.~Huang}, {\em Strong stability preserving hybrid methods}, Applied
  Numerical Mathematics, 59 (2009), pp.~891--904.

\bibitem{jin2005}
{\sc S.~Jin, H.~Liu, S.~Osher, and Y.-H.~R. Tsai}, {\em Computing multivalued
  physical observables for the semiclassical limit of the {S}chr{\"o}dinger
  equation}, Journal of Computational Physics, 205 (2005), pp.~222--241.

\bibitem{ketcheson2008}
{\sc D.~I. Ketcheson}, {\em Highly efficient strong stability preserving
  {R}unge--{K}utta methods with low-storage implementations}, SIAM Journal on
  Scientific Computing, 30 (2008), pp.~2113--2136.

\bibitem{ketcheson2009a}
{\sc D.~I. Ketcheson}, {\em {Computation of optimal monotonicity preserving
  general linear methods}}, Mathematics of Computation, 78 (2009),
  pp.~1497--1513.

\bibitem{tsrk}
{\sc D.~I. Ketcheson, S.~Gottlieb, and C.~B. Macdonald}, {\em Strong stability
  preserving two-step runge-kutta methods}, SIAM Journal on Numerical Analysis,
   (2012), pp.~2618--2639.

\bibitem{kraaijevanger1991}
{\sc J.~F. B.~M. Kraaijevanger}, {\em Contractivity of {R}unge--{K}utta
  methods}, BIT, 31 (1991), pp.~482--528.

\bibitem{labrunie2004}
{\sc S.~Labrunie, J.~Carrillo, and P.~Bertrand}, {\em Numerical study on
  hydrodynamic and quasi-neutral approximations for collisionless two-species
  plasmas}, Journal of Computational Physics, 200 (2004), pp.~267--298.

\bibitem{lenferink1989}
{\sc H.~W.~J. Lenferink}, {\em Contractivity-preserving explicit linear
  multistep methods}, Numerische Mathematik, 55 (1989), pp.~213--223.

\bibitem{VaillancourtAPNUM2011}
{\sc T.~Nguyen-Ba, H.~Nguyen-Thu, T.~Giordano, and R.~Vaillancourt}, {\em
  Strong-stability-preserving 3-stage {H}ermite-{B}irkhoff time-discretization
  methods}, Appl. Numer. Math., 61 (2011), pp.~487--500.

\bibitem{VaillancourtJSC2012}
{\sc T.~Nguyen-Ba, H.~Nguyen-Thu, T.~Giordano, and R.~Vaillancourt}, {\em
  Strong-stability-preserving 7-stage {H}ermite--{B}irkhoff time-discretization
  methods}, Journal of Scientific Computing, 50 (2012), pp.~63--90.

\bibitem{nguyen2011}
{\sc T.~Nguyen-Ba, H.~Nguyen-Thu, and R.~Vaillancourt}, {\em
  Strong-stability-preserving, k-step, 5-to 10-stage, {H}ermite-{B}irkhoff
  time-discretizations of order 12.}, American J. Computational Mathematics, 1
  (2011), pp.~72--82.

\bibitem{nguyenthesis}
{\sc H.~Nguyen-Thu}, {\em Strong-stability-preserving {H}ermite-{B}irkhoff
  time-discretization methods.}, Dissertation, University of Ottawa, Canada,
  (2012).

\bibitem{VaillancourtJCAM2014}
{\sc H.~Nguyen-Thu and R.~Nguyen-Ba, Truong~Vaillancourt}, {\em
  Strong-stability-preserving, {H}ermite {B}irkhoff time-discretization based
  on step methods and 8-stage explicit runge--kutta methods of order 5 and 4},
  Journal of Computational and Applied Mathematics, 263 (2014), pp.~45--58.

\bibitem{peng1999}
{\sc D.~Peng, B.~Merriman, S.~Osher, H.~Zhao, and M.~Kang}, {\em A {PDE}-based
  fast local level set method}, Journal of Computational Physics, 155 (1999),
  pp.~410--438.

\bibitem{ruuth2001}
{\sc S.~J. Ruuth and R.~J. Spiteri}, {\em Two barriers on
  strong-stability-preserving time discretization methods}, Journal of
  Scientific Computation, 17 (2002), pp.~211--220.

\bibitem{shu1988b}
{\sc C.-W. Shu}, {\em Total-variation diminishing time discretizations}, SIAM
  J. Sci. Stat. Comp., 9 (1988), pp.~1073--1084.

\bibitem{spijker2007}
{\sc M.~Spijker}, {\em Stepsize conditions for general monotonicity in
  numerical initial value problems}, SIAM Journal on Numerical Analysis, 45
  (2007), pp.~1226--1245.

\bibitem{spijker1983}
{\sc M.~N. Spijker}, {\em Contractivity in the numerical solution of initial
  value problems}, Numerische Mathematik, 42 (1983), pp.~271--290.

\bibitem{tanguay2003}
{\sc M.~Tanguay and T.~Colonius}, {\em Progress in modeling and simulation of
  shock wave lithotripsy ({SWL})}, in Fifth International Symposium on
  cavitation (CAV2003), 2003.

\end{thebibliography}

\end{document}